\documentclass[final, 12pt]{article}

\usepackage{amsmath}
\usepackage{amsfonts}
\usepackage{makeidx}
\usepackage{amsthm}
\usepackage{mathtools}
\usepackage{amssymb}
\usepackage{bm}
\usepackage{cite}
\usepackage{showkeys}
\usepackage[english]{babel}
\usepackage{dblaccnt}
\usepackage{accents}
\usepackage{graphicx}
\usepackage{subfig}
\usepackage{psfrag}
\usepackage{color}
\usepackage{float}
\usepackage{amscd}
\usepackage[mathscr]{euscript}
\usepackage{lipsum}
\usepackage{epstopdf}
\usepackage{algorithm}
\usepackage{algpseudocode}

\graphicspath{{images/}}

\newcommand{\Z}{\ensuremath{\mathbb{Z}}}
\newcommand{\R}{\ensuremath{\mathbb{R}}}
\newcommand{\C}{\ensuremath{\mathbb{C}}}

\renewcommand{\d}{\mathrm{d}}

\renewcommand{\Re}{\mathrm{Re}\,}

\newcommand{\E}{{\mathbf{E}}}

\renewcommand{\div}{\mathrm{div}}
\newcommand{\curl}{\mathrm{curl}\,}

\newcommand{\x}{\mathbf{x}}
\newcommand{\y}{\mathbf{y}}
\newcommand{\z}{\mathbf{z}}
\newcommand{\g}{\mathbf{g}}
\newcommand{\f}{\mathbf{f}}
\renewcommand{\u}{\mathbf{u}}
\renewcommand{\v}{\mathbf{v}}

\newtheorem{defi}{Definition}
\newtheorem{lemma}[defi]{Lemma}
\newtheorem{theorem}[defi]{Theorem}

\newtheorem{remark}[defi]{Remark}

\newcommand{\G}{\mathbb{G}}
\newcommand{\A}{\mathbf{A}}
\newcommand{\F}{\mathbb{F}}
\newcommand{\h}{\mathbf{h}}
\newcommand{\e}{\mathbf{e}}
\newcommand{\I}{\mathcal{I}}

\title{A stable imaging functional for anisotropic periodic media in electromagnetic inverse scattering}

\author{Dinh-Liem Nguyen \thanks{Department of Mathematics, Kansas State University, Manhattan, KS 66506, USA; (\texttt{dlnguyen@ksu.edu}, \texttt{trungt@ksu.edu})} \and Trung Truong \footnotemark[1] 
}

\begin{document}

\date{}
\maketitle

\begin{abstract}
 The paper is concerned with  the  inverse scattering problem for  Maxwell's equations 
 in three dimensional  anisotropic periodic media.  We study  a new imaging functional for fast and stable 
 reconstruction of the shape of  anisotropic periodic scatterers 
from boundary measurements of the scattered field for a number of  incident fields. 
This imaging functional is   simple to implement and very robust against noise in the data. Its implementation is non-iterative, computationally cheap, and
 does not involve solving any ill-posed problems. 
The resolution  and stability  analysis of the imaging functional is investigated. Our numerical study shows that this imaging functional is more stable than 
 that of the factorization method and more efficient 
 than  that of the orthogonality sampling method in reconstructing periodic scatterers.  
\end{abstract}

\sloppy


\bigskip


\section{Introduction}

We consider three dimensional periodic media that are unboundedly periodic  in 
say $x_1$- and $x_2$-directions and bounded in $x_3$-direction. These periodic media 
can model two-dimensional
photonic crystals that are popular in optics~\cite{Dorfl2012}. 
We are interested in the inverse problem of determining the shape of these  periodic media 
from boundary measurements of the scattered electromagnetic field generated by a number of incident fields.  
The motivation of this  inverse problem comes from
 applications of nondestructive evaluations for photonic crystals using electromagnetic waves. In the past two decades, there have
been a relatively large amount  of  results on 
numerical methods for shape reconstruction of periodic  media in inverse scattering. 
Results for the case  of Helmholtz type equations can be found in~\cite{Arens2003a, Elsch2003,  Arens2005, Lechl2010, Elsch2012,Yang2012, Zhang2014, Hadda2017, Zheng2017, Lechl2018, Cakon2019, Haddar2020, Nguye2020, Boukari2023, Nguye2023} and references therein. However, there have been only a limited number of results on numerical reconstructions for the inverse problem for full  Maxwell's equations in three dimensions, see~\cite{Sandf2010, Lechl2013b, Bao2014, Nguye2016, Jiang2017}. This is obviously due to the technical complication as well as the high computational complexity of Maxwell's equations in three-dimensional periodic media. The  methods that were mainly studied for the case of Maxwell's equations  are the factorization method~\cite{Sandf2010, Lechl2013b, Nguye2016} and the near field imaging method  that relies  on a transformed field expansion~\cite{Bao2014, Jiang2017, Bao2022}. The latter method can provide  subwavelength resolution but it requires  the periodic scattering  layer to be   a smooth periodic function multiplied by a small surface deformation parameter.   While the factorization method  is more flexible in terms of shape and properties (e.g. anisotropic or chiral) of periodic media, it is unfortunately not very robust against noise in the data.

In this paper we investigate a new imaging functional for the reconstruction of anisotropic periodic scatterers for Maxwell's equations in three dimensions. 
This new imaging functional is not only very robust against noise in the data but also quite flexible to different kinds of shape of periodic scattering media. 
The  implementation of the  imaging functional   is simple, computationally cheap, and fast as one only needs to evaluate a  double sum that essentially  involves a finite number of the  propagating modes of the scattered field data. The implementation also does not involve solving any ill-posed problems. 
We prove that the imaging functional is associated with a volume integral over the periodic scatterer in one unit cell and this volume integral has a kernel that strongly peaks as
the sampling point is inside the periodic scatterer. The stability of the imaging  functional is also established.   The numerical study also shows that the proposed sampling method is more stable than the factorization method and is more efficient than the  orthogonality sampling method in reconstructing periodic scattering media. This can be considered as 
an extended study  of  the result in~\cite{Nguye2023}, where the Helmholtz equation case was investigated. Due to the  technical complication of the Maxwell's
equations in periodic media  this extension is nontrivial and requires some innovations. The imaging functional for the Maxwell case  is an infinite series instead of a finite sum as in the scalar case. The extension requires  a  careful and detailed analysis for a modal version of 
Green   formulas for the quasiperiodic Green's tensor of the direct problem and   the quasiperiodic scattered electric field via its volume integro-differential formulation.
It is also worth noting  that although the orthogonality  sampling method has been studied for inverse scattering from bounded objects~\cite{Potth2010, Gries2011, Ito2013, Kang2018, Harri2020}, its application  to the inverse scattering problem for periodic media is still not known.

The rest of the paper is organized as follows. We formulate the  inverse problem of interest in Section~\ref{setup}. The new imaging functional and its resolution and stability analysis are discussed in Section~\ref{theory}.   A numerical study of the new imaging functional is presented in Section~\ref{numerical}.

\section{Electromagnetic  scattering from periodic media}
\label{setup}

We consider an anisotropic periodic medium in $\R^3$ that is unboundedly 2$\pi$-periodic  in 
 $x_1$- and $x_2$-directions and bounded in $x_3$-direction. Let $k>0$ be the wave number and $\varepsilon$ be a $3\times 3$ bounded matrix-valued function which represents the permittivity of the medium. We assume that $\varepsilon$ is $2\pi$-periodic in the $x_1$- and $x_2$-directions, and that in each period $(m_1\pi,(m_1+2)\pi)\times (m_2\pi,(m_2+2)\pi) \times \R$ for any $m_1,m_2 \in \Z$, it is equal to the identity matrix $I_3$ outside a compact set.
 This periodic medium is illuminated by an incident electric field $\E^{in}:\R^3 \to \C^3$, which is  generated by  source function  $\mathbf{J}: \R^3 \to \C^3 $. Their interaction gives rise to a scattered electric field $\u$ that is also  a function from $\R^3$ to $\C^3$. The total field $\E = \E^{in}+\u$ is assumed to satisfy the  Maxwell's equations
\begin{equation}
\label{pde1}
\curl \curl \E - k^2\varepsilon(\x) \E =  \mathbf{J}, \quad \x \in \R^3.
\end{equation}
 Now for $\alpha = (\alpha_1,\alpha_2) \in \R^2$, a function $\f :\R^3 \to \C^3$ is called $\alpha$-quasiperiodic if
$$
\f(x_1+m_12\pi,x_2+m_22\pi,x_3) = e^{2\pi i(\alpha_1m_1+\alpha_2m_2)}\f(x_1,x_2,x_3) \quad \text{for all } m_1, m_2 \in \Z.
$$
Fixing $\alpha = (\alpha_1,\alpha_2)$, we use $N$ $\alpha$-quasiperiodic incident electric fields to illuminate the periodic medium. Denote these incident fields by $\E^{in}(\cdot,l)$ for $l = 1,\dots N$. 
More specifically, $\E^{in}(\cdot,l)$ satisfy 
$$
\curl\curl \E^{in}(\cdot,l) - k^2\E^{in}(\cdot,l) = \mathbf{J}(\cdot, l), \quad l = 1,\dots N.
$$
 Following the usual approach   we look for $\alpha$-quasiperiodic scattered fields $\u(\cdot,l)$. We can 
 rewrite~\eqref{pde1} for $\u(\cdot,l)$ as follows
 \begin{align}
 \label{pde}
 \curl\curl \u(\cdot,l)  - k^2 \u(\cdot,l)  = k^2(\varepsilon(\x) - I_3)(\u(\cdot,l)  + \E^{in}(\cdot,l)).
 \end{align} 
 We complete the scattering problem with the Rayleigh radiation condition. To this end we first introduce some notations.
Let $h > 0$ be such that
$$
h > \sup\{ |x_3|: \x = (x_1,x_2,x_3)^\top \in \text{supp}(\varepsilon - I_3) \}.
$$
Now, for $\rho \geq h$, we denote  
$$
\Omega:= (-\pi,\pi)^2\times \R, \quad \Omega_h:= (-\pi,\pi)^2 \times (-h,h), \quad \Gamma_{\pm\rho}:=(-\pi,\pi)^2 \times \{ \pm\rho\}.
$$
In addition to  (\ref{pde}), the scattered electric fields $\u(\cdot,l)$ satisfy the Rayleigh radiation condition, i.e.,
\begin{equation}
\label{rc}
\u(\x,l) = \left\{ \begin{array}{ll} \sum_{j\in \Z^2} \u_j^+(l)e^{i(\alpha_{1,j}x_1 + \alpha_{2,j}x_2 + \beta_j(x_3-h)} & \text{if}\ x_3 > h, \\ \sum_{j\in \Z^2}\u_j^-(l)e^{i(\alpha_{1,j}x_1 + \alpha_{2,j}x_2 - \beta_j(x_3+h)} & \text{if}\ x_3 < -h, \end{array} \right.
\end{equation}
where
$$
\alpha_{1,j} = \alpha_1 + j_1, \quad \alpha_{2,j} = \alpha_2 + j_2, \quad
\alpha_j = (\alpha_{1,j},\alpha_{2,j},0)^\top 
$$
and \begin{equation}
\label{beta}
\beta_j = \left\{ \begin{array}{rl} \sqrt{k^2-|\alpha_{j}|^2},& |\alpha_j| \leq k,\\
i\sqrt{|\alpha_j|^2-k^2},&   |\alpha_j| > k, \end{array}\right.
\end{equation}
for all $j = (j_1,j_2) \in \Z^2$. 
The sequences of coefficients $(\u_j^+(l))_j$ and $(\u_j^-(l))_j$ are called the Rayleigh sequences of $\u(\cdot,l)$ and can be computed by
\begin{equation}
\label{ray_coeff}
\u_j^{\pm}(l) = \frac{1}{4\pi^2} \int_{\Gamma_{\pm r}} \u(\x,l)e^{-i(\alpha_{1,j}x_1 + \alpha_{2,j}x_2 \pm \beta_j(x_3\mp h))}\d s(\x)
\end{equation}
for any $r \geq h$.

Note that all but finitely many terms in (\ref{rc}) are exponentially decaying, which helps us easily deduce pointwise absolute convergence of the series.
The exponentially decaying terms in  (\ref{rc})  are called evanescent modes and the terms corresponding to real $\beta_j$'s are called propagating modes.  Moreover, we need $\beta_j$ to be nonzero for all $j \in\mathbb{Z}^2$ or $k$ is not a Wood's anomaly.  The technical reason behind this assumption is that the representation of the $\alpha$-quasiperiodic  Green's function we use in~\eqref{green}  is not well-defined at a Wood's anomaly. 
For the study of the inverse problem in this paper  we assume that the direct problem~\eqref{pde}-\eqref{rc}  is well-posed. 
We refer to~\cite{Bao2022, Nguye2015} for studies on well-posedness of the direct problem~\eqref{pde}-\eqref{rc}. 

Note that since the scattering medium is $2\pi$-periodic in $x_1$ and $x_2$, finding its geometry in $\Omega$ is sufficient. Denote by $D$ the geometry of the medium in $\Omega$, i.e.
$$
D:= \text{supp}(\varepsilon - I_3) \cap \Omega.
$$
We aim to solve the following inverse problem.\\

\noindent
\textbf{Inverse Problem.} Given measurement of multiple scattered electric fields $\u(\cdot,l)$ on $\Gamma_{\pm\rho}$ corresponding to multiple incident fields $\E^{in}(\cdot,l)$, $l= 1,\dots, N$, find $D$ in $\Omega_h$.


\section{A new imaging functional}
\label{theory}
We will introduce a new imaging functional and analyze its behavior in this section. To this end we first need   the quasiperiodic Green's tensor of the direct problem. 
For $\x-\y \neq (2\pi m_1,2\pi m_2,0)$, for all $ m_1,m_2 \in \Z$,  the $\alpha$-quasiperiodic Green's tensor of the direct problem is given by (see, e.g., \cite{Sandf2010})
$$
\G(\x,\y) := \Phi(\x,\y)I_3 + \frac{1}{k^2} \nabla_{\x} \div_{\x} (\Phi(\x,\y)I_3),
$$
where the divergence is  columnwise and the gradient is componentwise, and  $\Phi(\x,\y)$ is the $\alpha$-quasiperiodic Green's function of the scalar Helmholtz problem
\begin{equation}
\label{green}
\Phi(\x,\y):= \frac{i}{8\pi^2}\sum_{j\in\Z^2} \frac{1}{\beta_j}e^{i(\alpha_{1,j}(x_1-y_1) + \alpha_{2,j}(x_2-y_2) + \beta_j|x_3-y_3|)},
\end{equation}
with Rayleigh coefficients given by 
\begin{align}
\label{ray}
r_j^{\pm}(\y) = \frac{i}{8\pi^2\beta_j}e^{-i\alpha_{1,j}y_1 -i\alpha_{2,j}y_2 + i\beta_j(h \mp y_3)}.
\end{align}



\begin{lemma}
\label{lem1}
For $\z \in \Omega_h$, the columns of $\G(\cdot,\z)$ satisfy the Rayleigh radiation condition (\ref{rc}). Let $G_{mn}(\cdot,\z)$ be the entry on the $m$-th row, $n$-th column of $\G(\cdot,\z)$. Then, the Rayleigh sequences $(g^{\pm}_{mn})_j(\z)$ of $G_{mn}(\cdot,\z)$ can be given as an expression in terms of $r_j^{\pm}(\z)$,
$$
(g^{\pm}_{mn})_j(\z) = \left( \delta_{mn} - \frac{\gamma^{\pm}_{m,j}\gamma^{\pm}_{n,j}}{k^2} \right) r^{\pm}_j(\z),   \quad  j \in \Z^2,
$$
where $\delta_{mn}$ is the Kronecker delta and
$$
\gamma^{\pm}_{n,j} := \left\{ \begin{array}{ll} \alpha_{n,j} & \text{if}\ n=1,2, \\ \pm \beta_j & \text{if}\ n=3. \end{array} \right.
$$
\end{lemma}

\begin{proof}
For $\x \in \Omega$ with $x_3 > h$ and $\z\in\Omega_h$,
$$
\Phi(\x,\z) = \sum_{j\in \Z^2} r_j^+(\z)e^{i(\alpha_{1,j}x_1 + \alpha_{2,j}x_2 + \beta_j(x_3-h))},
$$
thus,
\begin{align*}
\G(\x,\z) &= \Phi(\x,\y)I_3 + \frac{1}{k^2} \nabla_{\x} \div_{\x} (\Phi(\x,\y)I_3) \\
&= \sum_{j\in \Z^2}\left( \begin{bmatrix} 1 & 0 & 0 \\ 0 & 1 & 0 \\ 0 & 0 & 1 \end{bmatrix} + \frac{1}{k^2}\begin{bmatrix} \frac{\partial^2}{\partial x_1^2} & \frac{\partial^2}{\partial x_1\partial x_2} & \frac{\partial^2}{\partial x_1\partial x_3} \\ \frac{\partial^2}{\partial x_2\partial x_1}  & \frac{\partial^2}{\partial x_2^2} & \frac{\partial^2}{\partial x_2\partial x_3} \\ \frac{\partial^2}{\partial x_3\partial x_1} & \frac{\partial^2}{\partial x_3\partial x_2} & \frac{\partial^2}{\partial x_3^2} \end{bmatrix} \right)r_j^+(\z) e^{i(\alpha_{1,j}x_1 + \alpha_{2,j}x_2 + \beta_j(x_3-h))} \\
&= \sum_{j\in \Z^2}\left( \begin{bmatrix} 1 & 0 & 0 \\ 0 & 1 & 0 \\ 0 & 0 & 1 \end{bmatrix} - \frac{1}{k^2}\begin{bmatrix} \alpha_{1,j}^2 & \alpha_{1,j}\alpha_{2,j} & \alpha_{1,j}\beta_j \\ \alpha_{2,j}\alpha_{1,j}  & \alpha_{2,j}^2 & \alpha_{2,j}\beta_j \\ \beta_j\alpha_{1,j} & \beta_j\alpha_{2,j} & \beta_j^2 \end{bmatrix} \right)r_j^+(\z) e^{i(\alpha_{1,j}x_1 + \alpha_{2,j}x_2 + \beta_j(x_3-h))} \\
&= \sum_{j\in \Z^2}\left[ \delta_{mn} - \frac{\gamma^+_{m,j}\gamma^+_{n,j}}{k^2} \right]_{mn}r_j^+(\z) e^{i(\alpha_{1,j}x_1 + \alpha_{2,j}x_2 + \beta_j(x_3-h))}.
\end{align*}
Similarly, we can show that, for $x\in \Omega$ such that $x_3 < -h$,
$$
\G(\x,\z) = \sum_{j\in \Z^2}\left[ \delta_{mn} - \frac{\gamma^-_{m,j}\gamma^-_{n,j}}{k^2} \right]_{mn}r_j^-(\z) e^{i(\alpha_{1,j}x_1 + \alpha_{2,j}x_2 - \beta_j(x_3+h))}.
$$
This shows that, each entry $G_{mn}(\cdot,\z)$ of $\G(\cdot,\z)$ admits a Rayleigh series representation for $x_3 > h$ or $x_3 < -h$. Therefore, the columns of $\G(\cdot,\z)$ satisfy the Rayleigh radiation condition and the Rayleigh coefficients of $G_{mn}(\cdot,\z)$ are given by
$$
(g^{\pm}_{mn})_j(\z) = \left( \delta_{mn} - \frac{\gamma^{\pm}_{m,j}\gamma^{\pm}_{n,j}}{k^2} \right) r^{\pm}_j(\z).
$$
\end{proof}
For $j \in  \Z^2$, we  denote
$$
\g_j^{\pm}(\z) := \left[ (g^\pm_{mn})_j(\z)  \right]_{mn}, \quad
\h_j^{\pm}(\z) := \begin{bmatrix} \alpha_{1,j}(g^\pm_{31})_j &  \alpha_{1,j}(g^\pm_{32})_j & \alpha_{1,j}(g^\pm_{33})_j  \\ \alpha_{2,j}(g^\pm_{31})_j &  \alpha_{2,j}(g^\pm_{32})_j & \alpha_{2,j}(g^\pm_{33})_j \\ \pm\beta_j (g^\pm_{31})_j &  \pm\beta_j(g^\pm_{32})_j & \pm\beta_j(g^\pm_{33})_j \end{bmatrix}(\z).
$$
Now with the data $\u(\cdot,l)$ given on $\Gamma_{\pm\rho}$ for $l = 1, \dots, N$, we also know the Rayleigh coefficients $(\u_j^{\pm}(l))_j$ via~\eqref{ray_coeff}. 
We define the imaging  functional as
\begin{multline*}
\I(\z) := \sum_{l = 1}^N\left| \sum_{j\in\Z^2} \left( \left( \h_j^+(\z)  - 2\Re(\beta_j)\ \g^+_j(\z)\right)^*\u_j^+(l)  + u^+_{3,j}(l)\g^+_j(\z)^*\begin{bmatrix} \alpha_{1,j} \\ \alpha_{2,j} \\ \beta_j \end{bmatrix} \right. \right. \\
\left. \left. - \left( \h_j^-(\z)  + 2\Re(\beta_j)\ \g^-_j(\z)\right)^*\u_j^-(l)  - u^-_{3,j}(l) \g^-_j(\z)^* \begin{bmatrix} \alpha_{1,j} \\ \alpha_{2,j} \\ -\beta_j \end{bmatrix} \right) \right|^p,
\end{multline*}
where $\A^*$ denotes the transpose conjugate of the matrix $\A$. Here $p > 0$ is chosen to sharpen the reconstruction of the imaging functional (e.g., $p = 3$ works well in the numerical simulations.) 

\begin{remark}
\label{remark1}
 We note that the Rayleigh sequences $(r_j^{\pm}(\z))_j$ in~\eqref{ray}  are exponentially decaying as $|j|$ increases and $\beta_j$ is complex-valued.  For $m,n = 1,2,3$, the sequences $\left( (g^{\pm}_{mn})_j(\z)\right)_j$ involve $r_j^{\pm}(\z)$ multiplied by $\beta_j$ and $\alpha_{1,j}$ or $\alpha_{2,j}$. Thus these sequences    are quickly  decaying as $|j|$ increases and $\beta_j$ is complex-valued. This property  holds for all $\z \in \Omega_h$.
This leads to the fact that $\g^{\pm}_j(\z) $ and $\h^{\pm}_j(\z)$ also have quickly decaying entries for complex-valued $\beta_j$'s. Therefore, $\mathcal{I}(\z)$ is well-defined  and that only a finite number of  terms corresponding to real-valued $\beta_j$'s may make significant contributions to  $\mathcal{I}(\z)$ and the terms corresponding to complex-valued $\beta_j$'s can be essentially ignored. This is confirmed  in the numerical study. 
\end{remark}

Define the Sobolev spaces $H_\alpha(\curl,\Omega_h)$ and $H_{\alpha,\text{loc}}(\curl,\Omega)$ as
\begin{align*}
H_\alpha(\curl,\Omega_h) &:= \left\{ \mathbf{w} \in [L^2(\Omega_h)]^3: \curl \mathbf{w} \in [L^2(\Omega_h)]^3 \ \text{and}\ \mathbf{w} = \mathbf{W}|_{\Omega_h}\ \text{for some $\alpha$-quasiperiodic $\mathbf{W}$} \right\}, \\
H_{\alpha,\text{loc}}(\curl,\Omega) &= \left\{ \mathbf{w} \in [L^2_{\text{loc}}(\Omega)]^3: \curl \mathbf{w} \in [L^2_{\text{loc}}(\Omega)]^3 \ \text{and}\ \mathbf{w} = \mathbf{W}|_{\Omega}\ \text{for some $\alpha$-quasiperiodic $\mathbf{W}$} \right\}.
\end{align*}
The unique weak solution $\u \in H_{\alpha,\text{loc}}(\curl,\Omega)$ of the direct problem (\ref{pde})-(\ref{rc}) satisfies the Rayleigh radiation condition (\ref{rc}) and
\begin{equation}\label{var}
\int_\Omega \curl \u(\x) \cdot \curl \overline{\v(\x)} - k^2 \u(\x) \cdot \overline{\v(\x)}\, \d\x = k^2\int_D(\varepsilon(\x) - I_3)(\u(\x) + \E^{in}(\x))\cdot \overline{\v(\x)}\, \d\x,
\end{equation}
for all $\v \in H_\alpha(\curl,\Omega_h)$ with compact support.  It is known that $\u$  also satisfies 
 the volume integro-differential equation (see, e.g., \cite{Nguye2015})
\begin{equation}\label{ls}
\u(\x) = (k^2 + \nabla_\x\div_\x)\int_D \Phi(\x,\y) (\varepsilon(\y) - I_3)\E(\y)\,\d\y.
\end{equation}
 The equivalence is understood in the sense that, if $\u \in H_{\alpha,\text{loc}}(\curl,\Omega)$ satisfies (\ref{rc}) and (\ref{var}) then $\u|_{\Omega_h}$ belongs to $H_\alpha(\curl,\Omega_h)$ and solves (\ref{ls}), and conversely, if $\u|_{\Omega_h} \in H_\alpha(\curl,\Omega_h)$ solves (\ref{ls}) then it can be extended into a solution of (\ref{var}) in $H_{\alpha,\text{loc}}(\curl,\Omega)$ that also satisfies the radiation condition (\ref{rc}).

We study the resolution of the imaging functional in the following theorem.

\begin{theorem} 
\label{resolution}
The imaging functional $\mathcal{I}(\z)$ satisfies
$$
\I(\z) = \sum_{l = 1}^N \left|  \frac{k^2}{2\pi^2} \int_D \F(\z,\y)\,  (\varepsilon(\y) - I_3)\E(\y,l)\, \d\y \right|^p, \quad \z \in \Omega,
$$
where
$$
\F(\x,\y) = F(\x,\y)I_3 + \frac{1}{k^2} \nabla_{\x} \div_{\x}(F(\x,\y)I_3),
$$
with 
$$
F(\x,\y) = \frac{k}{4\pi} j_0(k|\x-\y|) +  \sum_{ 0 \neq  j \in \Z^2}   \frac{ke^{-i2\pi\alpha \cdot j}}{4\pi}  j_0\left(k \left | \x-\y + (2\pi j_1, 2\pi j_2, 0)^\top \right| \right)   
$$
($j_0$ is the spherical Bessel function of the first kind of order $0$).

\end{theorem}

\begin{remark}
For a fixed $\y$, we numerically observe that  the series in $F(\x,\y)$ makes a relatively small perturbation to $\frac{k}{4\pi} j_0(k|\x-\y|)$ that  strongly peaks as $\x$ is close to $\y$ and has much smaller values otherwise. Thus the behavior of  $|F(\x,\y)|$ is pretty  similar to that of $\frac{k}{4\pi} j_0(k|\x-\y|)$ as it can be seen in Figure~\ref{j0F}. A two-dimensional version of $F(\x,\y)$ was studied in~\cite{Nguye2023} with similar behaviors.  We can thus expect $\F(\x,\y)$ to have a similar behavior. Figure \ref{Fp} shows the values of $|\F(\x,0)\mathbf{q}|_2$ when $k = 2\pi$, $\alpha_1 = \alpha_2 = 0$ and for  $\mathbf{q} = (0,0,1)^\top$ and  $\mathbf{q} = (1,1,1)^\top$. In both cases, $|\F(\x,0)\mathbf{q}|_2$ behaves as expected. We also note that similar behaviors of $|\F(\x,0)\mathbf{q}|_2$ were observed for  $\mathbf{q} = (1,0,0)^\top$ and  $\mathbf{q} = (0,1,0)^\top$.  Therefore, we expect from Theorem~\ref{resolution} that the imaging functional $\I(\z)$ takes larger values as $\z$ is inside $D$ and that $\I(\z)$ is much smaller  for $\z$ is outside $D$.
This is indeed  confirmed in the numerical study.
\end{remark}

\begin{figure}[h!]
\centering
\subfloat[ $|\frac{k}{4\pi}j_0(k|\x|)|$]{\includegraphics[width=6cm]{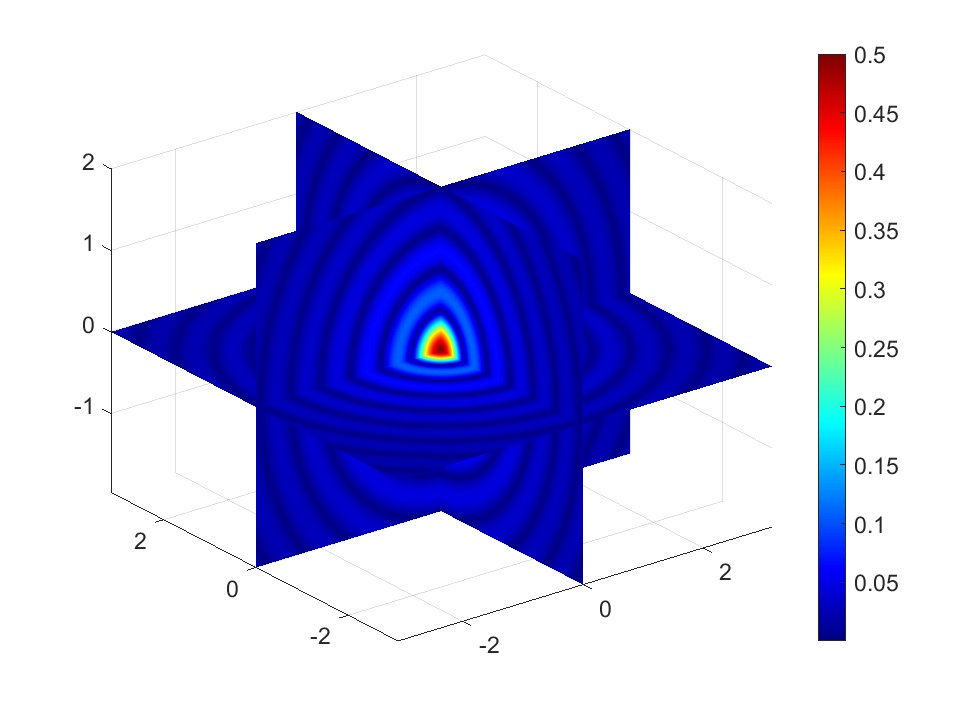}}
\subfloat[$|F(\x,0)|$]{\includegraphics[width=6cm]{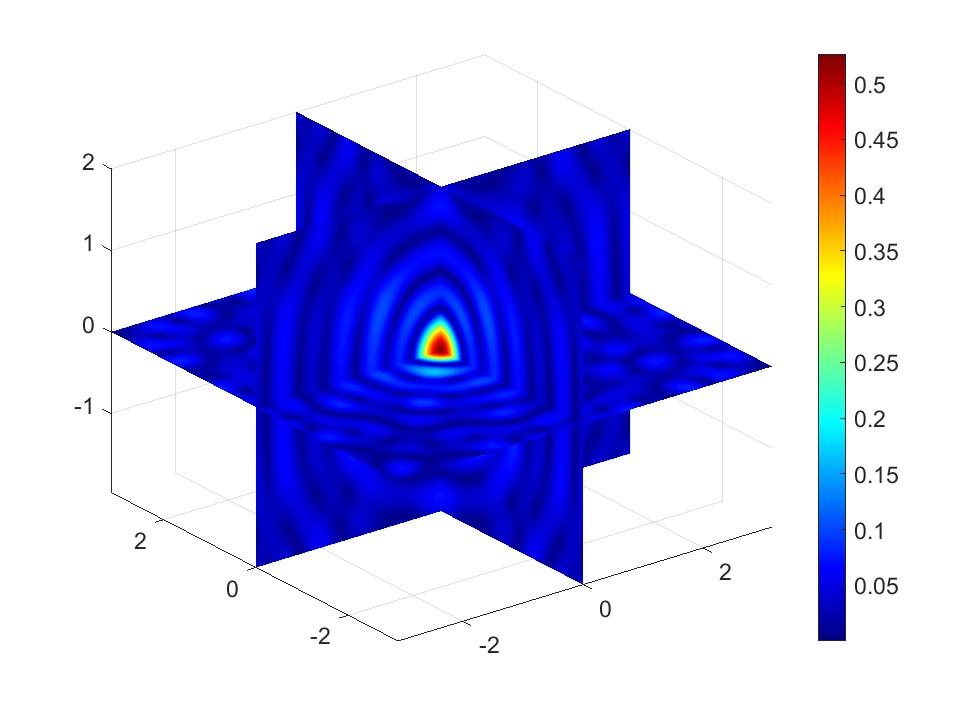}}\\
\subfloat[Top view of (a) ]{\includegraphics[width=6cm]{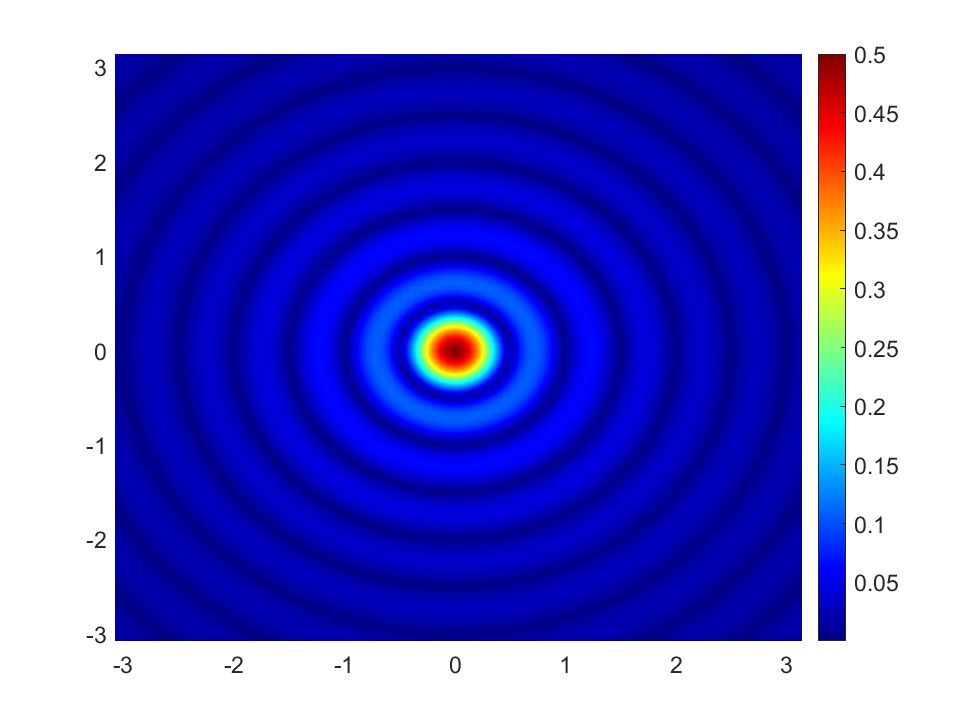}}
\subfloat[ Top view of (b)]{\includegraphics[width=6cm]{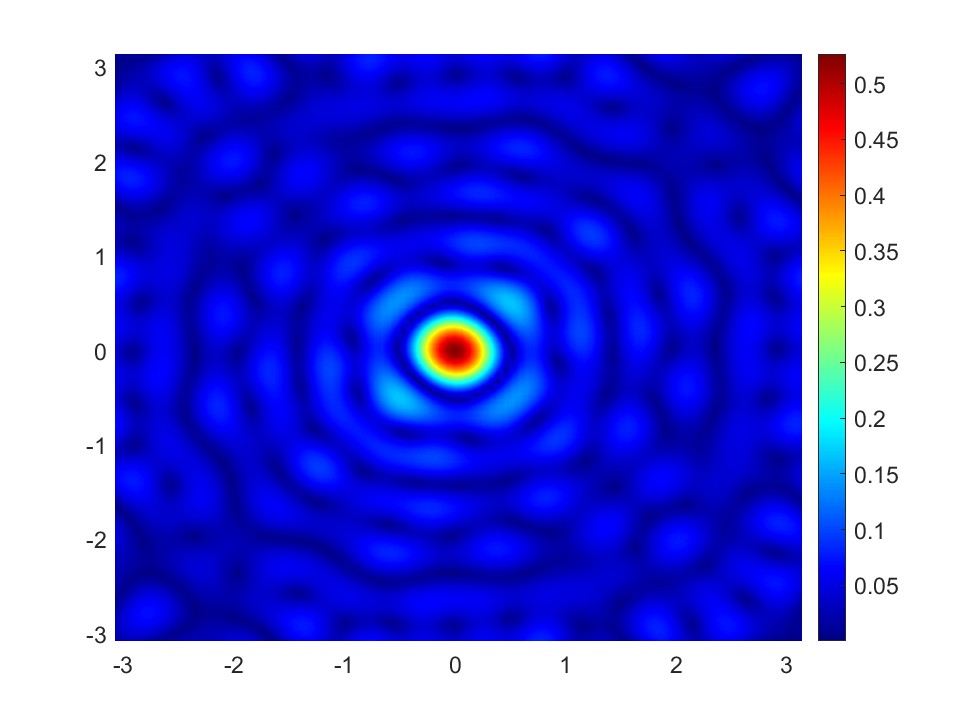}}
\caption{$\left|\frac{k}{4\pi}j_0(k|\x|)\right|$ and $|F(\x,0)|$, $k = 2\pi$, $\alpha_1=\alpha_2=0$.}
\label{j0F}
\end{figure}

\begin{figure}[h!]
\centering
\subfloat[ $\mathbf{q} =(0,0,1)^\top$]{\includegraphics[width=6cm]{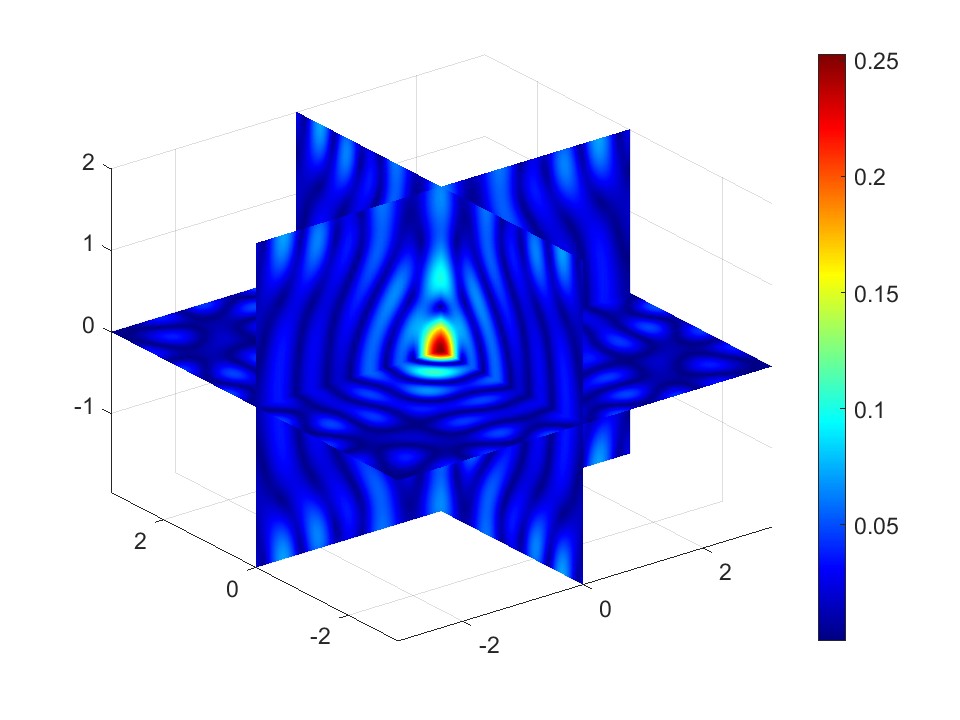}}
\subfloat[$\mathbf{q} =(1,1,1)^\top$]{\includegraphics[width=6cm]{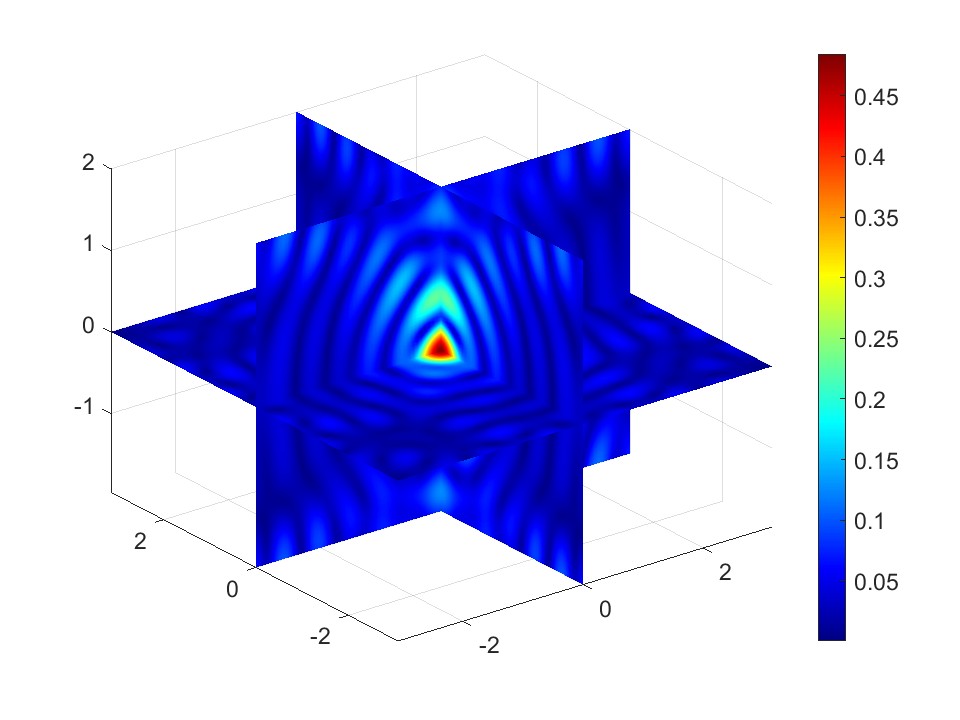}}\\
\subfloat[Top view of (a) ]{\includegraphics[width=6cm]{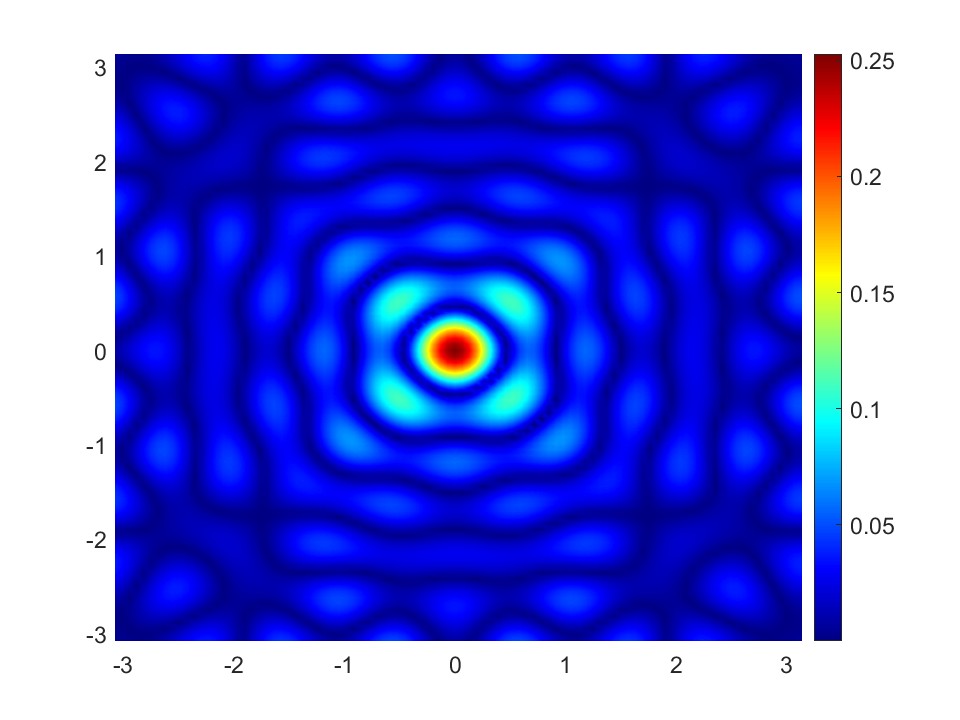}}
\subfloat[ Top view of (b)]{\includegraphics[width=6cm]{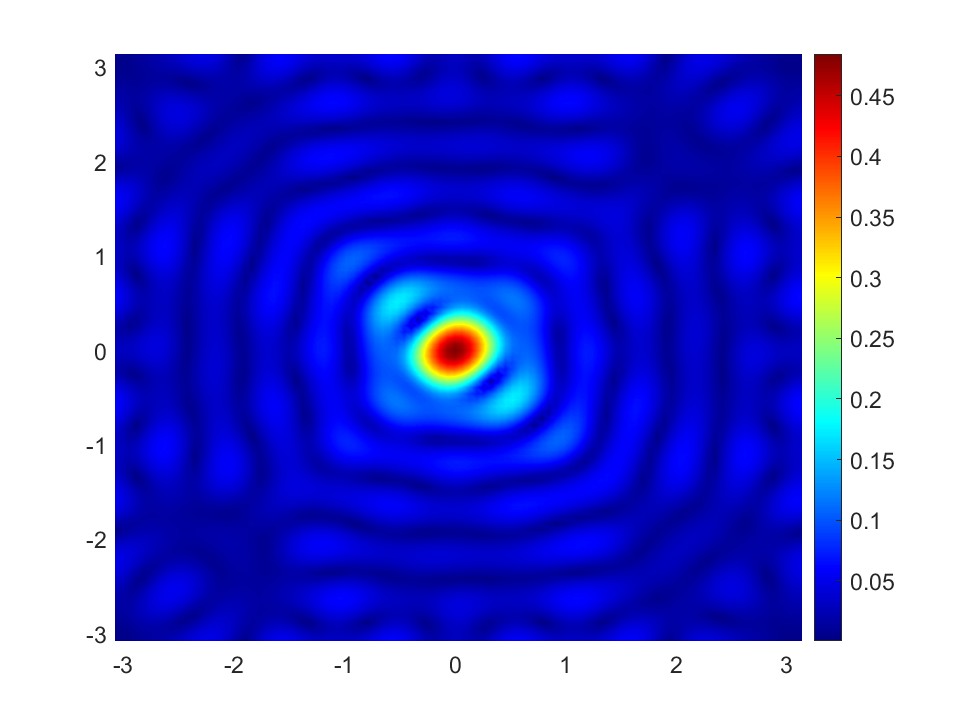}}
\caption{$|\F(\x,0)\mathbf{q}|_2$ for different vectors $\mathbf{q}$, $k = 2\pi$, $\alpha_1=\alpha_2=0$.}
\label{Fp}
\end{figure}

\begin{proof}
Let $\mathbf{G}_n$ be the $n$-th column of $\G$, $n = 1,2,3$. For all $\x_s,\x_t \in \Omega_h$ and $\x \in \overline{\Omega_h}$, we have
$$
\curl\curl \mathbf{G}_n(\x,\x_s) -k^2 \mathbf{G}_n(\x,\x_s) = \delta(\x-\x_s)\e_n,
$$
and dot-multiplying both sides by $\overline{\mathbf{G}_m(\x,\x_t)}$, $m = 1,2,3$ gives
$$
\curl \curl \mathbf{G}_n(\x,\x_s)\cdot \overline{\mathbf{G}_m(\x,\x_t)} -k^2 \mathbf{G}_n(\x,\x_s)\cdot \overline{\mathbf{G}_m(\x,\x_t)} = \delta(\x-\x_s)\e_n\cdot \overline{\mathbf{G}_m(\x,\x_t)}.
$$
Integrating by part over $\Omega_h$ with respect to $\x$ we obtain
\begin{multline}\label{hki:1}
\int_{\Omega_h} \curl \mathbf{G}_n(\x,\x_s)\cdot \curl \overline{\mathbf{G}_m(\x,\x_t)}\, \d\x - k^2\int_{\Omega_h}\mathbf{G}_n(\x,\x_s)\cdot \overline{\mathbf{G}_m(\x,\x_t)}\, \d\x \\
 + \int_{\partial\Omega_h}\nu(\x) \times \curl \mathbf{G}_n(\x,\x_s)\cdot \overline{\mathbf{G}_m(\x,\x_t)} \, \d s(\x) = \overline{G}_{nm}(\x_s,\x_t).
\end{multline}
Similarly, $\overline{\mathbf{G}_m(\x,\x_t)}$, $m = 1,2,3$ satisfies
$$
\curl \curl \overline{\mathbf{G}_m(\x,\x_t)} -k^2 \overline{\mathbf{G}_m(\x,\x_t)} = \delta(\x-\x_t)\e_m,
$$
and dot-multiplying both sides with $\mathbf{G}_n(\x,\x_s)$, $n = 1,2,3$ gives
$$
\curl \curl \overline{\mathbf{G}_m(\x,\x_t)}\cdot \mathbf{G}_n(\x,\x_s) -k^2 \overline{\mathbf{G}_m(\x,\x_t)}\cdot \mathbf{G}_n(\x,\x_s) = \delta(\x-\x_t)\e_m\cdot \mathbf{G}_n(\x,\x_s).
$$
Integrating by part over  $\Omega_h$ with respect to $\x$ leads to
\begin{multline}\label{hki:2}
\int_{\Omega_h} \curl \overline{\mathbf{G}_m(\x,\x_t)}\cdot \curl \mathbf{G}_n(\x,\x_s)\, \d\x - k^2\int_{\Omega_h}\overline{\mathbf{G}_m(\x,\x_t)}\cdot \mathbf{G}_n(\x,\x_s)\, \d\x \\
 + \int_{\partial\Omega_h}\nu(\x) \times \curl \overline{\mathbf{G}_m(\x,\x_t)}\cdot \mathbf{G}_n(\x,\x_s) \, \d s(\x) = G_{mn}(\x_t,\x_s).
\end{multline}
Subtracting (\ref{hki:1}) from (\ref{hki:2}) we obtain
\begin{multline*}
\int_{\partial\Omega_h} \nu(\x) \times \curl \overline{\mathbf{G}_m(\x,\x_t)}\cdot \mathbf{G}_n(\x,\x_s) - \nu(\x) \times \curl \mathbf{G}_n(\x,\x_s)\cdot \overline{\mathbf{G}_m(\x,\x_t)} \, \d s(\x) \\ = G_{mn}(\x_t,\x_s) - \overline{G_{nm}(\x_s,\x_t)},
\end{multline*}
or in matrix form
\begin{align}
\label{HK}
& \int_{\partial\Omega_h} (\nu(\x) \times \curl \G(\x,\x_t))^* \G(\x,\x_s) - \G(\x,\x_t)^* \nu(\x) \times \curl \G(\x,\x_s) \, \d s(\x) \nonumber \\
& = \G(\x_t,\x_s) - \G^*(\x_s,\x_t),
\end{align}
{\color{black}
where $\curl$ and $\times$ are taken columnwise.
}
Since $\G$ is $\alpha$-quasiperiodic, the integral on the left-hand side can be taken on just $\Gamma_{+h}\cup \Gamma_{-h}$.
Recall from Lemma~\ref{lem1} that the columns of $\G(\cdot,\y)$ satisfy the Rayleigh radiation condition. Letting
$$
\phi^{\pm}_j(\x) := e^{i(\alpha_{1,j}x_1 + \alpha_{2,j}x_2 \pm \beta_j(x_3 \mp h))},
$$ 
we compute
\begin{align*}
&\int_{\Gamma_{\pm h}}   (\nu(\x) \times \curl \G(\x,\x_t))^* \G(\x,\x_s) - \G(\x,\x_t)^* \nu(\x) \times \curl \G(\x,\x_s) \, \d s(\x)\\
&= \int_{\Gamma_{\pm h}} \left( \begin{bmatrix} 0 \\ 0 \\ \pm 1 \end{bmatrix} \times \sum_{j \in \Z^2} \curl \left[ \g^{\pm}_j(\x_t)\phi^{\pm}_j(\x) \right] \right)^*\sum_{j \in \Z^2} \g^{\pm}_j(\x_s)\phi^{\pm}_j(\x)\ \d s(\x) \\
&- \int_{\Gamma_{\pm h}} \left( \sum_{j \in \Z^2} \g^{\pm}_j(\x_t)\phi^{\pm}_j(\x) \right)^* \begin{bmatrix} 0 \\ 0 \\ \pm 1 \end{bmatrix} \times \sum_{j \in \Z^2} \curl \left[ \g^{\pm}_j(\x_s)\phi^{\pm}_j(\x) \right]\ \d s(\x).
\end{align*}
For any matrix $\A$ (independent of $\x$),  we obtain  from a direct calculation that
\begin{align*}
\curl \left[\A \phi^{\pm}_j(\x) \right] 
&= i\begin{bmatrix} a_{31}\alpha_{1,j} \mp a_{21}\beta_j  & a_{32}\alpha_{1,j} \mp a_{22}\beta_j  & a_{33}\alpha_{1,j} \mp a_{23}\beta_j \\
\pm a_{11} \beta_j  - a_{31}\alpha_{1,j}   &  \pm a_{12} \beta_j  - a_{32}\alpha_{1,j}  & \pm a_{13} \beta_j  - a_{33}\alpha_{1,j}  \\
a_{21}\alpha_{1,j}  - a_{11}\alpha_{2,j}  & a_{22}\alpha_{1,j}  - a_{12}\alpha_{2,j}  & a_{23}\alpha_{1,j}  - a_{13}\alpha_{2,j}  \\\end{bmatrix} \phi^{\pm}_j(\x) \\
&=  \left( i\begin{bmatrix} \alpha_{1,j} \\ \alpha_{2,j} \\ \pm \beta_j \end{bmatrix} \times \A \right)\phi^{\pm}_j(\x).
\end{align*}
Hence, combining with dominated convergence theorem, we have
\begin{align*}
&\int_{\Gamma_{\pm h}}   (\nu(\x) \times \curl \G(\x,\x_t))^* \G(\x,\x_s) - \G(\x,\x_t)^* \nu(\x) \times \curl \G(\x,\x_s) \ \d s(\x)\\
&= \int_{\Gamma_{\pm h}} \sum_{j \in \Z^2}   \left(  \begin{bmatrix} 0 \\ 0 \\ \pm 1 \end{bmatrix} \times \left( i\begin{bmatrix} \alpha_{1,j} \\ \alpha_{2,j} \\ \pm \beta_j \end{bmatrix} \times \g^{\pm}_j(\x_t) \right) \right)^*\overline{\phi^{\pm}_j(\x)}\sum_{j \in \Z^2} \g^{\pm}_j(\x_s)\phi^{\pm}_j(\x)\ \d s(\x) \\
&- \int_{\Gamma_{\pm h}} \sum_{j \in \Z^2}\g^{\pm}_j(\x_t)^* \overline{\phi^{\pm}_j(\x)} \sum_{j \in \Z^2} \left(\begin{bmatrix} 0 \\ 0 \\ \pm 1 \end{bmatrix} \times \left( i\begin{bmatrix} \alpha_{1,j} \\ \alpha_{2,j} \\ \pm \beta_j \end{bmatrix} \times \g^{\pm}_j(\x_s) \right) \right)\phi^{\pm}_j(\x) \ \d s(\x) \\
&= \sum_{j\in \Z^2}\sum_{j' \in \Z^2} \left(  \begin{bmatrix} 0 \\ 0 \\ \pm 1 \end{bmatrix} \times \left( i\begin{bmatrix} \alpha_{1,j} \\ \alpha_{2,j} \\ \pm \beta_j \end{bmatrix} \times \g^{\pm}_j(\x_t) \right) \right)^*\g^{\pm}_{j'}(\x_s) \int_{\Gamma_{\pm h}} \phi^{\pm}_{j'}(\x)\overline{\phi^{\pm}_j(\x)}\ \d s(\x) \\
&- \sum_{j \in \Z^2}\sum_{j' \in \Z^2}\g^{\pm}_j(\x_t)^* \left(\begin{bmatrix} 0 \\ 0 \\ \pm 1 \end{bmatrix} \times \left( i\begin{bmatrix} \alpha_{1,j'} \\ \alpha_{2,j'} \\ \pm \beta_{j'} \end{bmatrix} \times \g^{\pm}_{j'}(\x_s) \right) \right) \int_{\Gamma_{\pm h}}   \phi^{\pm}_{j'}(\x) \overline{\phi^{\pm}_j(\x)} \ \d s(\x).
\end{align*}
Note that using
$$
\int_{\Gamma_{\pm h}} \phi^{\pm}_{j'}(\x)\overline{\phi^{\pm}_j(\x)}\ \d s(\x) = \left\{ \begin{array}{ll} 4\pi^2 & \text{if}\ j = j',\\ 0 & \text{if}\ j \neq j', \end{array} \right.
$$
we have
\begin{align*}
&\int_{\Gamma_{\pm h}}   (\nu(\x) \times \curl \G(\x,\x_t))^* \G(\x,\x_s) - \G(\x,\x_t)^* \nu(\x) \times \curl \G(\x,\x_s) \ \d s(\x)\\
&= 4\pi^2  \sum_{j\in\Z^2}\left( \left(  \begin{bmatrix} 0 \\ 0 \\ \pm 1 \end{bmatrix} \times \left( i\begin{bmatrix} \alpha_{1,j} \\ \alpha_{2,j} \\ \pm \beta_j \end{bmatrix} \times \g^{\pm}_j(\x_t) \right) \right)^* \g^{\pm}_j(\x_s) \right. \\
&\left. - \g^{\pm}_j(\x_t)^*  \left(\begin{bmatrix} 0 \\ 0 \\ \pm 1 \end{bmatrix} \times \left( i\begin{bmatrix} \alpha_{1,j} \\ \alpha_{2,j} \\ \pm \beta_j \end{bmatrix} \times \g^{\pm}_j(\x_s) \right) \right) \right) \\
&= -i4\pi^2\sum_{j\in\Z^2}\left( \left( \pm \h_j^{\pm}(\x_t)  - \beta_j \g^{\pm}_j(\x_t) \right)^* \g^{\pm}_j(\x_s) + \g^{\pm}_j(\x_t)^* \left( \pm\h_j^{\pm}(\x_s)  - \beta_j \g^{\pm}_j(\x_s) \right) \right) \\
&= -i4\pi^2\sum_{j\in\Z^2}\left( \left( \pm \h_j^{\pm}(\x_t)  - 2\Re(\beta_j)\ \g^{\pm}_j(\x_t)  \right)^* \g^{\pm}_j(\x_s)  \pm \g^{\pm}_j(\x_t)^* \h_j^{\pm}(\x_s) \right).
\end{align*}
Therefore we obtain that
\begin{align}
\label{MHK}
&\int_{\partial \Omega_h}   (\nu(\x) \times \curl \G(\x,\x_t))^* \G(\x,\x_s) - \G(\x,\x_t)^* \nu(\x) \times \curl \G(\x,\x_s) \d s(\x) \nonumber \\
&= -i4\pi^2\sum_{j\in\Z^2}\left( \left( \h_j^+(\x_t)  - 2\Re(\beta_j)\ \g^+_j(\x_t)  \right)^*\g^+_j(\x_s) + \g^+_j(\x_t)^* \h_j^+(\x_s)  \right. \nonumber \\
&\phantom{2\pi^2\sum_{j\in\Z^2}\quad\ }- \left. \left( \h_j^-(\x_t)  + 2\Re(\beta_j)\ \g^-_j(\x_t)  \right)^*\g^-_j(\x_s) - \g^-_j(\x_t)^* \h_j^-(\x_s) \right).
\end{align}
Now recall that the  scattered field $\u(\x,l)$ satisfies
\begin{equation*}
\u(\x,l) = (k^2 + \nabla_\x\div_\x)\int_D \Phi(\x,\y) (\varepsilon(\y) - I_3)\E(\y,l)\,\d\y.
\end{equation*}
Thus by (\ref{ray_coeff}) we compute
\begin{align*}
\u_j^{\pm}(l) &= \frac{1}{4\pi^2} \int_{\Gamma_{\pm r}} \u(\x,l)e^{-i(\alpha_{1,j}x_1 + \alpha_{2,j}x_2 \pm \beta_j(x_3\mp h))}\d s(\x) \\
&=  \frac{1}{4\pi^2} \int_{\Gamma_{\pm r}}(k^2 + \nabla_\x\div_\x)\int_D \Phi(\x,\y) (\varepsilon(\y) - I_3)\E(\y,l)\,\d\y\ e^{-i(\alpha_{1,j}x_1 + \alpha_{2,j}x_2 \pm \beta_j(x_3\mp h))}\d s(\x) \\
&= \frac{1}{4\pi^2} \int_{\Gamma_{\pm r}} \int_D (k^2 + \nabla_\x\div_\x) \left( \Phi(\x,\y) (\varepsilon(\y) - I_3)\E(\y,l) \right)   \d\y\ e^{-i(\alpha_{1,j}x_1 + \alpha_{2,j}x_2 \pm \beta_j(x_3\mp h))}\d s(\x).
\end{align*}
Note that for any vector field $\v(\y)$,
\begin{align*}
&(k^2 + \nabla_{\x} \div_{\x})\left[ \Phi(\x,\y)\v(\y) \right] \\
&= k^2\Phi(\x,\y)\v(\y) + \begin{bmatrix} \frac{\partial^2}{\partial x_1^2}\Phi(\x,\y)v_1(\y) + \frac{\partial^2}{\partial x_1\partial x_2}\Phi(\x,\y)v_2(\y) + \frac{\partial^2}{\partial x_1 \partial x_3}\Phi(\x,\y)v_3(\y) \\  \frac{\partial^2}{\partial x_2\partial x_1}\Phi(\x,\y)v_1(\y) + \frac{\partial^2}{\partial x_2^2}\Phi(\x,\y)v_2(\y) + \frac{\partial^2}{\partial x_2 \partial x_3}\Phi(\x,\y)v_3(\y) \\ \frac{\partial^2}{\partial x_3\partial x_1}\Phi(\x,\y)v_1(\y) + \frac{\partial^2}{\partial x_3\partial x_2}\Phi(\x,\y)v_2(\y) + \frac{\partial^2}{\partial x_3^2}\Phi(\x,\y)v_3(\y) \end{bmatrix} \\
&=  k^2\Phi(\x,\y)\v(\y) + \nabla_{\x} \div_{\x} (\Phi(\x,\y)I_3) \v(\y) \\
&= k^2\G(\x,\y)\v(\y).
\end{align*}
Thus, along with Fubini's theorem, we have
\begin{align}
\u_j^{\pm}(l)&= \frac{k^2}{4\pi^2} \int_{\Gamma_{\pm r}} \int_D \G(\x,\y) (\varepsilon(\y) - I_3)\E(\y,l)\,\d\y \, e^{-i(\alpha_{1,j}x_1 + \alpha_{2,j}x_2 \pm \beta_j(x_3\mp h))}\d s(\x) \nonumber \\
&= k^2 \int_D \frac{1}{4\pi^2} \int_{\Gamma_{\pm r}}\G(\x,\y)e^{-i(\alpha_{1,j}x_1 + \alpha_{2,j}x_2 \pm \beta_j(x_3\mp h))} \d s(\x)\, (\varepsilon(\y) - I_3)\E(\y,l)\,\d\y  \nonumber \\
&= k^2\int_{D}\g_j^{\pm}(\y)  (\varepsilon(\y) - I_3)\E(\y,l)\, \d\y,
\label{coeff_u}
\end{align}
and so
\begin{align}
\begin{bmatrix} \alpha_{1,j} \\ \alpha_{2,j} \\ \pm \beta_j \end{bmatrix}u^{\pm}_{3,j}(l) &= k^2 \begin{bmatrix} \alpha_{1,j} \\ \alpha_{2,j} \\ \pm \beta_j \end{bmatrix} \int_D  \begin{bmatrix} (g^{\pm}_{31})_j & (g^{\pm}_{32})_j & (g^{\pm}_{33})_j\end{bmatrix}(\y) (\varepsilon(\y) - I_3)\E(\y,l)\, \d\y \nonumber \\
&= k^2 \int_D \begin{bmatrix} \alpha_{1,j}(g^\pm_{31})_j &  \alpha_{1,j}(g^\pm_{32})_j & \alpha_{1,j}(g^\pm_{33})_j  \\ \alpha_{2,j}(g^\pm_{31})_j &  \alpha_{2,j}(g^\pm_{32})_j & \alpha_{2,j}(g^\pm_{33})_j \\ \pm\beta_j (g^\pm_{31})_j &  \pm\beta_j(g^\pm_{32})_j & \pm\beta_j(g^\pm_{33})_j \end{bmatrix} (\varepsilon(\y) - I_3)\E(\y,l)\, \d\y \nonumber \\
&=   k^2 \int_{D}\h_j^{\pm}(\y)  (\varepsilon(\y) - I_3)\E(\y,l)\, \d\y.
\label{coeff_u3}
\end{align}
Now recall that 
\begin{multline*}
\I(\z) = \sum_{l = 1}^N\left| \sum_{j\in\Z^2} \left( \left( \h_j^+(\z)  - 2\Re(\beta_j)\ \g^+_j(\z)\right)^*\u_j^+(l)  + u^+_{3,j}(l)\g^+_j(\z)^*\begin{bmatrix} \alpha_{1,j} \\ \alpha_{2,j} \\ \beta_j \end{bmatrix} \right. \right. \\
\left. \left. - \left( \h_j^-(\z)  + 2\Re(\beta_j)\ \g^-_j(\z)\right)^*\u_j^-(l)  - u^-_{3,j}(l) \g^-_j(\z)^* \begin{bmatrix} \alpha_{1,j} \\ \alpha_{2,j} \\ -\beta_j \end{bmatrix} \right) \right|^p.
\end{multline*}
Plugging the formula of $\u_j^{\pm}(l)$ and  $[\alpha_{1,j} \ \alpha_{2,j} \ \pm \beta_j ]^\top u^{\pm}_{3,j}(l)$ in~\eqref{coeff_u} and \eqref{coeff_u3}  into $\mathcal{I}(\z)$, and using identities~\eqref{HK} and \eqref{MHK} we obtain that 
$$
\I(\z) = \sum_{l = 1}^N \left|  \frac{k^2}{2\pi^2} \int_D\frac{1}{2i}\left( \G(\z,\y) - \G^*(\y,\z)\right)  (\varepsilon(\y) - I_3)\E(\y,l)\, \d\y \right|^p.
$$
Now letting
$$ 
F(\x,\y) := \frac{1}{2i} \left( \Phi(\x,\y) - \overline{\Phi(\y,\x)} \right),
$$
we calculate 
\begin{align*}
\F(\x,\y) &:= \frac{1}{2i}(\G(\x,\y) - \G^*(\y,\x)) \\
&= F(\x,\y)I_3 + \frac{1}{2ik^2}\left( \nabla \div(\Phi(\x,\y)I_3) -  \left(\nabla \div(\Phi(\y,\x)I_3)\right)^*\right) \\
&= F(\x,\y)I_3 + \frac{1}{2ik^2}\begin{bmatrix} \frac{\partial^2}{\partial x_i \partial x_j}\Phi(\x,\y) -  \overline{\frac{\partial^2}{\partial x_j \partial x_i}\Phi(\y,\x)} \end{bmatrix}_{ij} \\
&= F(\x,\y)I_3 + \frac{1}{k^2}\begin{bmatrix} \frac{\partial^2}{\partial x_i \partial x_j}F(\x,\y) \end{bmatrix}_{ij}\\
&= F(\x,\y)I_3 + \frac{1}{k^2}\nabla \div(F(\x,\y)I_3).
\end{align*}
The proof follows from using the following expression of the scalar $\alpha$-quasiperiodic Green's function $\Phi(\x,\y)$
$$
\Phi(\x,\y) = \sum_{j\in \Z^2}e^{-i2\pi \alpha\cdot j}\frac{e^{ik\sqrt{(x_1-y_1+2j_1\pi)^2 + (x_2-y_2+2j_2\pi)^2 + (x_3-y_3)^2}}}{4\pi \sqrt{(x_1-y_1+2j_1\pi)^2 + (x_2-y_2+2j_2\pi)^2 + (x_3-y_3)^2}},
$$
for $\x,\y \in \Omega_h$ such that $x_3 \neq y_3$. 
\end{proof}

%

Denote by  $|\cdot|_F$ and $|\cdot|_2$  the Frobenius norm of a matrix and the $2$-norm of a vector respectively.
In the next theorem we analyze the stability of the imaging functional.

\begin{theorem}
(Stability) For $\delta > 0$, let $\u^{\delta}$ be the noisy data such that
$$
\sum_{l=1}^N  \left\| \u^{\delta}(\cdot,l) - \u(\cdot,l) \right\|_{L^2(\Gamma_\rho \cup \Gamma_{-\rho})} \leq \delta,
$$
and let $\mathcal{I}_{\delta}$ be the imaging functional computed from this data. Then, for all $\z \in \Omega_h$,
$$
|\mathcal{I}_{\delta}(\z) - \mathcal{I}(\z)| = O(\delta).
$$
\end{theorem}

{\color{black}
\begin{proof}
Note that, for all $j \in \Z^2$ and $l=1,\dots,N$,
$$
\left| \u_{\delta,j}^{\pm}(l) - \u_j^{\pm}(l) \right|_2 \leq \frac{1}{4\pi^2}\int_{\Gamma_{\pm \rho}}\left| \u^{\delta}(\x,l) - \u(\x,l) \right|\,\d s(\x),
$$
thus, by Cauchy-Schwarz inequality,
$$
\sum_{l=1}^N \left| \u_{\delta,j}^{\pm}(l) - \u_j^{\pm}(l) \right|_2 \leq \sum_{l=1}^N  \left\| \u^{\delta}(\cdot,l) - \u(\cdot,l) \right\|_{L^2(\Gamma_{\pm\rho} )} \leq \delta. 
$$
We first prove the theorem when $p = 1$. Using Cauchy-Schwarz inequality and triangle inequality, we estimate
\begin{align*}
\left| \mathcal{I}_{\delta}(\z) - \mathcal{I}(\z) \right| &\leq \sum_{l=1}^N \sum_{j \in \Z^2} \left| \u_{\delta,j}^+(l) - \u_j^+(l) \right|_2 \left( \left| \h_j^+(\z) \right|_F + 2\Re(\beta_j) \left| \g_j^+(\z) \right|_F + \left| \g_j^+(\z)^* \begin{bmatrix} \alpha_{1,j} \\ \alpha_{2,j} \\ \beta_j \end{bmatrix}  \right|_2 \right) \\
&+ \sum_{l=1}^N \sum_{j \in \Z^2} \left| \u_{\delta,j}^-(l) - \u_j^-(l) \right|_2 \left( \left| \h_j^-(\z) \right|_F + 2\Re(\beta_j) \left| \g_j^-(\z) \right|_F + \left|  \g_j^-(\z)^* \begin{bmatrix} \alpha_{1,j} \\ \alpha_{2,j} \\ -\beta_j \end{bmatrix} \right|_2 \right) \\
&\leq  \sum_{j \in \Z^2} \left( \left( \left| \h_j^+(\z) \right|_F + (2\Re(\beta_j) + k) \left| \g_j^+(\z) \right|_F  \right) \sum_{l=1}^N \left| \u_{\delta,j}^+(l) - \u_j^+(l) \right|_2 \right) \\
&+  \sum_{j \in \Z^2} \left( \left( \left| \h_j^-(\z) \right|_F + (2\Re(\beta_j) + k) \left| \g_j^-(\z) \right|_F  \right)  \sum_{l=1}^N \left| \u_{\delta,j}^-(l) - \u_j^-(l) \right|_2 \right)\\
&\leq C_1 \delta,
\end{align*}
where 
$$
C_1:= \sum_{j \in \Z^2} \sup_{\z \in \Omega_h} \left( \left| \h_j^+(\z) \right|_F + \left| \h_j^-(\z) \right|_F + (2\Re(\beta_j) + k) \left( \left| \g_j^+(\z) \right|_F +  \left| \g_j^-(\z) \right|_F \right)  \right).
$$
Note that $C_1 < \infty$ since the sequences $\left(\sup_{\z \in \Omega_h} \left| \h_j^{\pm}(\z) \right|_F\right)_j$ and $\left(\sup_{\z \in \Omega_h} \left| \g_j^{\pm}(\z) \right|_F\right)_j$  quickly decay as mentioned before.

For $p \geq 2$, using triangle inequality, we estimate
\begin{align*}
\left| \mathcal{I}_{\delta}(\z) - \mathcal{I}(\z) \right| &= \left| \left( \mathcal{I}_{\delta}(\z)^{1/p} - \mathcal{I}(\z)^{1/p} \right) \sum_{m=1}^{p} \mathcal{I}_{\delta}(\z)^{m/p} \mathcal{I}(\z)^{(p-1-m)/p}\right| \\
&\leq \left| \mathcal{I}_{\delta}(\z)^{1/p} - \mathcal{I}(\z)^{1/p} \right| \sum_{m=1}^{p} \mathcal{I}_{\delta}(\z)^{m/p} \mathcal{I}(\z)^{(p-1-m)/p} \\
&\leq  C_1 \delta \sum_{m=1}^{p-1} \left( \left| \mathcal{I}(\z)^{1/p} - \mathcal{I}_{\delta}(\z)^{1/p} \right| + \mathcal{I}(\z)^{1/p} \right)^m \mathcal{I}(\z)^{(p-1-m)/p}.
\end{align*}
Note that, for two nonnegative numbers $a$ and $b$,
$$
(a + b)^m \leq (2\max\{a,b\})^m = 2^m \max\{a^m,b^m\} \leq 2^m(a^m + b^m),\quad \text{for all } m=1,\dots, p,
$$
hence,
\begin{align*}
\left| \mathcal{I}_{\delta}(\z) - \mathcal{I}(\z) \right| &\leq   C_1 \delta \sum_{m=1}^{p-1} 2^m \left( \left| \mathcal{I}(\z)^{1/p} - \mathcal{I}_{\delta}(\z)^{1/p} \right|^m + \mathcal{I}(\z)^{m/p} \right) \mathcal{I}(\z)^{(p-1-m)/p} \\
&\leq C_1 \delta \sum_{m=1}^{p-1} 2^m \left( C_1^m \delta^m + \mathcal{I}(\z)^{m/p} \right) \mathcal{I}(\z)^{(p-1-m)/p} \\
&= O(\delta),
\end{align*}
which completes the proof.
\end{proof}
}

\section{Numerical study}
\label{numerical}
We   tested the performance of the new imaging functional for data with and without evanescent modes, for  different levels of noise in the data, for different number of incident sources, and we also compare with the orthogonality sampling method for different types of periodic media.
We used $k = 2\pi$ and $\alpha_1 = \alpha_2 = 0$ in the numerical simulation. The exponential of the imaging functional is $p=3$. Choosing $h=1$, the domain $\Omega_h = (-\pi,\pi)^2 \times (-1,1)$ where the medium is sought is partitioned into a $40 \times 40$ grid. To simplify the calculation we choose $\mathbf{J}(\cdot,l)  = \sum_{j \in \Z^2} (0,0,1)^\top \delta_{\y_l + 2\pi j}$  that means the incident fields are the third column of the Green's tensor, more specifically, they are emitted from point sources and have the form
\begin{equation}
\label{inc}
\E^{inc} (\x,l) = -\frac{i}{8\pi^2k^2}\sum_{j \in \Z^2} \begin{bmatrix} \text{sgn}(y^l_3 - x_3) \alpha_{1,j} \\ \text{sgn}(y^l_3 - x_3) \alpha_{2,j} \\ k^2/\beta_j + \beta_j  \end{bmatrix} e^{i(\alpha_{1,j}(y^l_1 - x_1) + \alpha_{2,j}(y^l_2 - x_2) + \beta_j |y^l_3 - x_3|)}, \quad \x \in \Omega_h
\end{equation}
where $\y_l = (y_1^l,y_2^l,y_3^l)$ are the sources' locations and $\text{sgn}$ is the sign function. The sources are placed evenly on two planes $(-\pi,\pi)^2 \times \{\pm 2.5\}$. Except the test with different number of incident sources in section~\ref{source},
we use $450$ incident sources to illuminate the periodic medium, with $225$ sources on $(-\pi,\pi)^2 \times \{2.5\}$ and $225$ sources on $(-\pi,\pi)^2 \times \{-2.5\}$.

To generate data for the inverse problem, we solved the volume integro-differential equation~\eqref{ls} using a spectral Galerkin method studied in~\cite{Nguye2015}. Instead of first solving for the scattered fields on $\Gamma_{\pm\rho}$ then computing their Rayleigh coefficients using (\ref{ray_coeff}), we solved directly for their Rayleigh coefficients to save computational time. We tested and observed that this does not make a difference in the numerical results for the inverse problem. After getting the Rayleigh coefficients $\u^{\pm}_j(l)$, we added artificial noise with noise level $\delta > 0$. To be specific, let $\mathbf{U}$ be the matrix containing all $\u^{\pm}_j(l)$, we created a matrix $\mathbf{N}$ of the same size as $\mathbf{U}$. The entries of $\mathbf{N}$ are complex numbers whose real and imaginary parts are uniformly distributed random numbers on $[-1,1]$. The noisy version of $\mathbf{U}$ is then
$$
\mathbf{U}_\delta := \mathbf{U} + \delta |\mathbf{U}|_F \frac{\mathbf{N}}{|\mathbf{N}|_F}.
$$
For all examples except those in section \ref{sub:noise}, the noise level is $\delta = 20\%$. 
In all of the numerical examples the isovalue for the plotting of 3D reconstructions is chosen as $60\%$ of the maximal value of 
the imaging functional $\I(\z)$.
 We consider the following types of periodic media. 
\paragraph{Rings.} The first periodic medium consists of a short hollow cylinder which resembles a ring in each period. The inner circle has radius $1$ while the outer circle has radius $1.5$. The height of the cylinder is $0.2$. The permittivity $\varepsilon(\x)$ is given by
$$
\varepsilon(\x) = \left\{ \begin{array}{ll} \text{diag}(1.3,1.5,1.4) & \text{if}\ \x \in D, \\ I_3 & \text{if}\ \x \notin D. \end{array} \right.
$$

\paragraph{Spheres.} The second periodic medium that we considered consists of four aligned spheres in each period. The radius of each sphere is $0.4$, and the permittivity is similar to that of the ring case. 

\paragraph{Cubes.} The third periodic medium that we considered consists of one cube in each period. The size of each cube is $(-1,1)^2 \times (-0.3,0.3)$, and the permittivity is similar to that of the ring case.

\subsection{Reconstruction with  and without evanescent modes (Figure~\ref{evan})}
The imaging  functional is an infinite sum, however, there are only finitely many terms whose indices are such that $\beta_j$ is real-valued (or $k^2 > |\alpha_j|^2$). These terms are associated with propagating modes. The rest of the terms are associated with  evanescent modes and correspond to complex-valued $\beta_j$'s.  As discussed in Remark~\ref{remark1} we observed numerically that the terms associated with evanescent modes do not contribute much to the reconstruction, which means we can compute the imaging functional using only propagating modes and the results will still be the same. See Figure \ref{evan} for a comparison between two reconstructions for the same periodic medium, one with and one without evanescent modes.
\begin{figure}[h!]
\centering
\subfloat[]{\includegraphics[width=5.5cm]{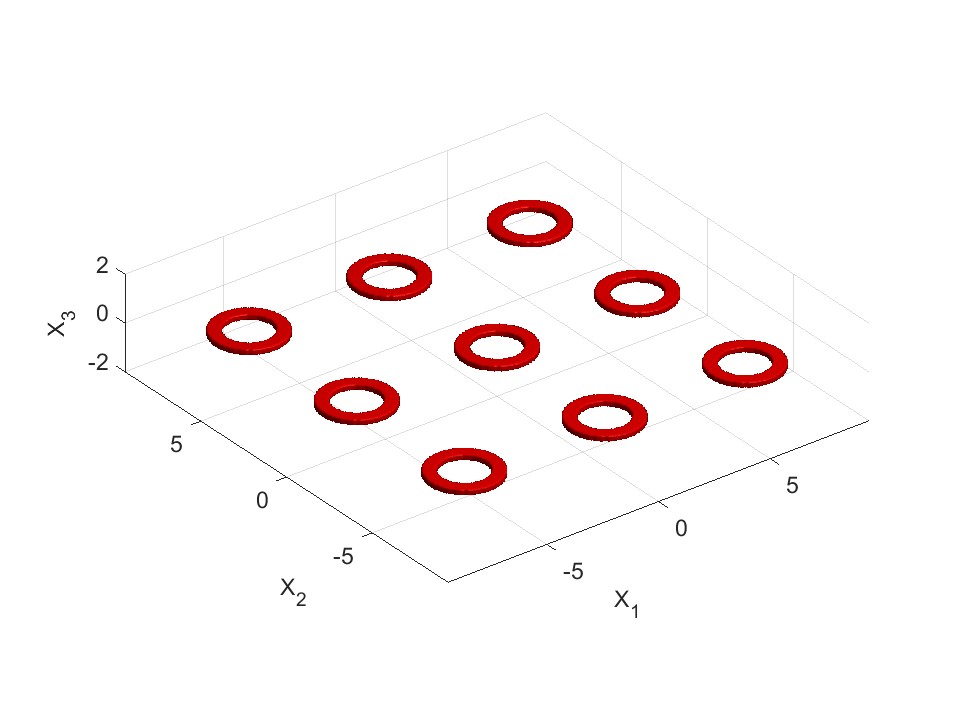}}\hspace{-0.4cm}
\subfloat[]{\includegraphics[width=5.5cm]{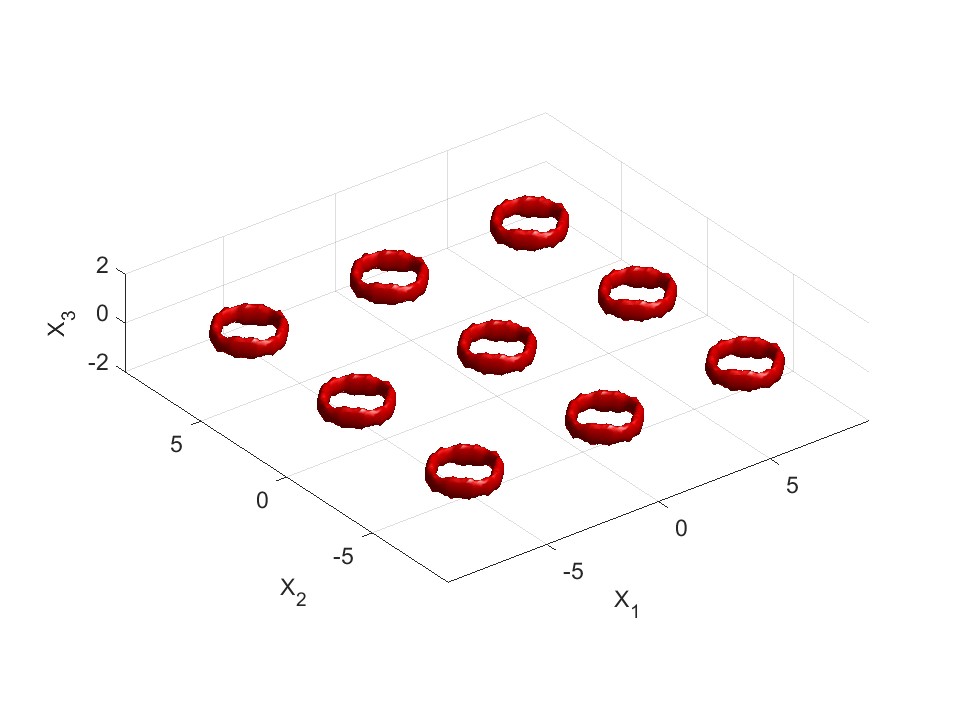}} \hspace{-0.4cm}
\subfloat[]{\includegraphics[width=5.5cm]{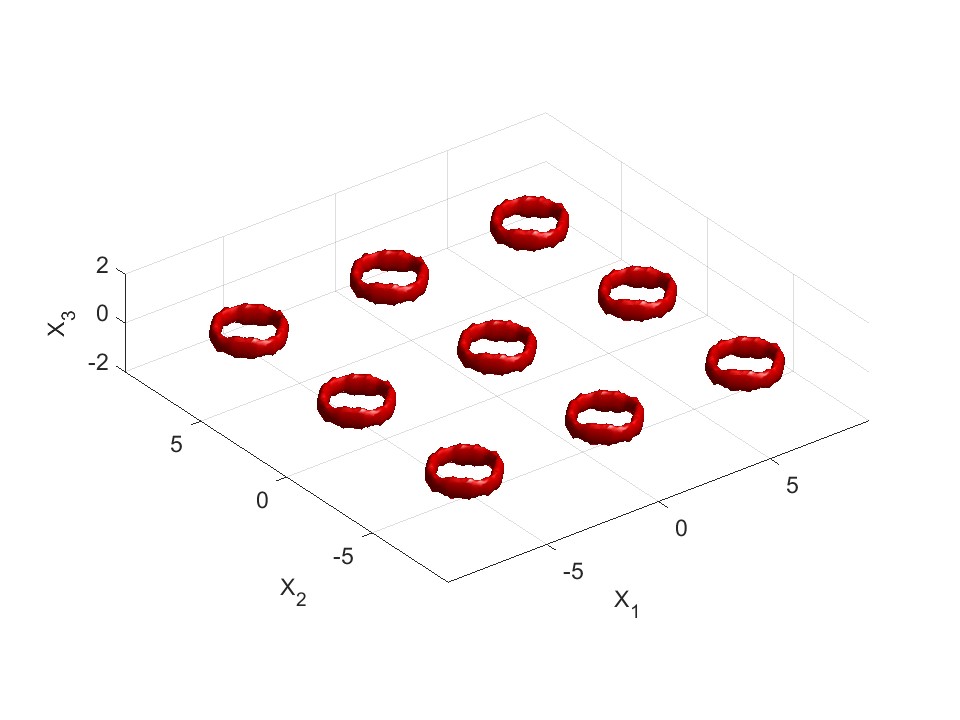}} \\
\subfloat[]{\includegraphics[width=5.5cm]{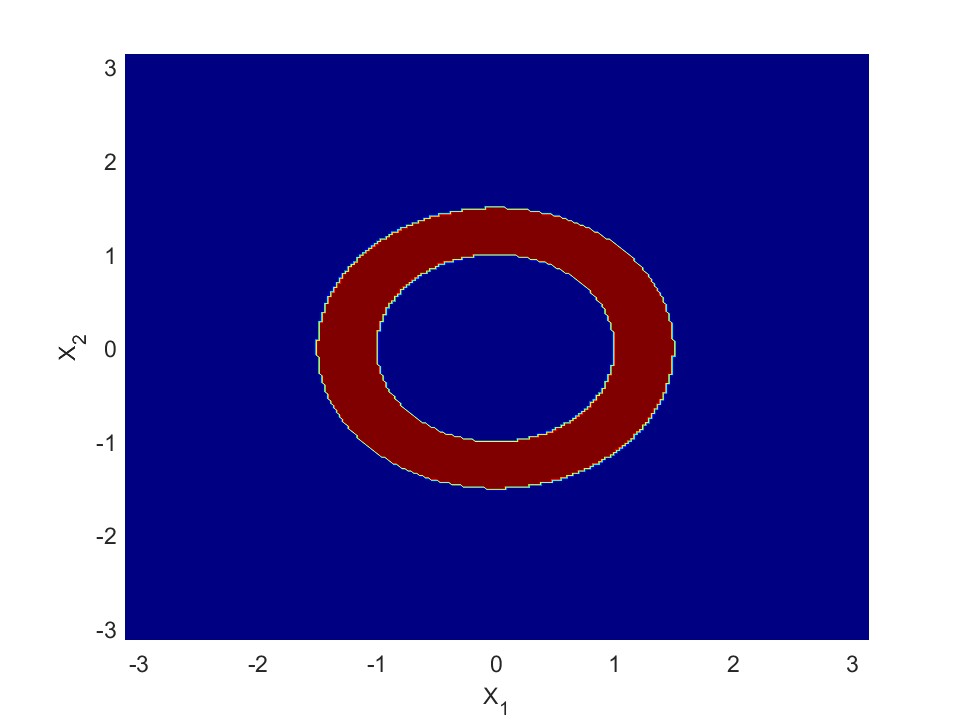}} \hspace{-0.3cm}
\subfloat[]{\includegraphics[width=5.5cm]{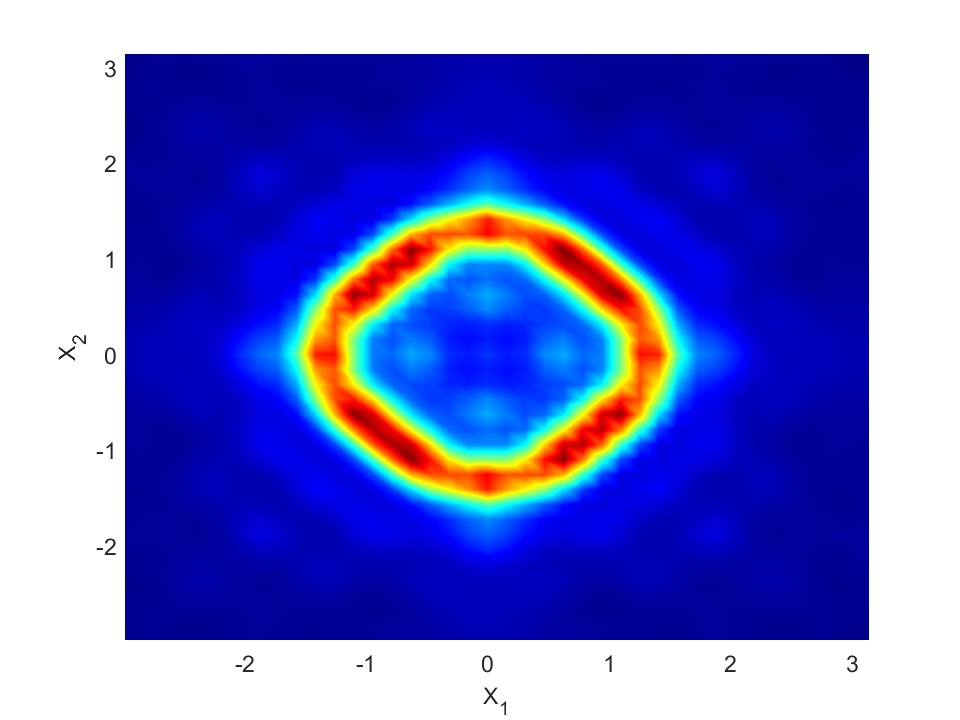}} \hspace{-0.3cm}
\subfloat[]{\includegraphics[width=5.5cm]{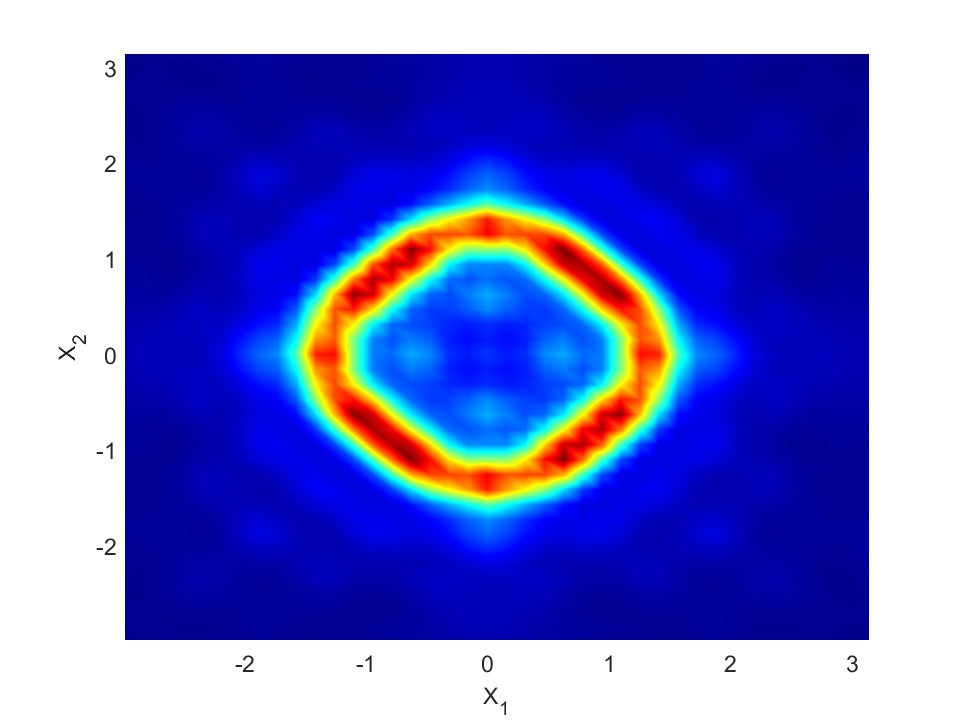}}\\
\subfloat[]{\includegraphics[width=5.5cm]{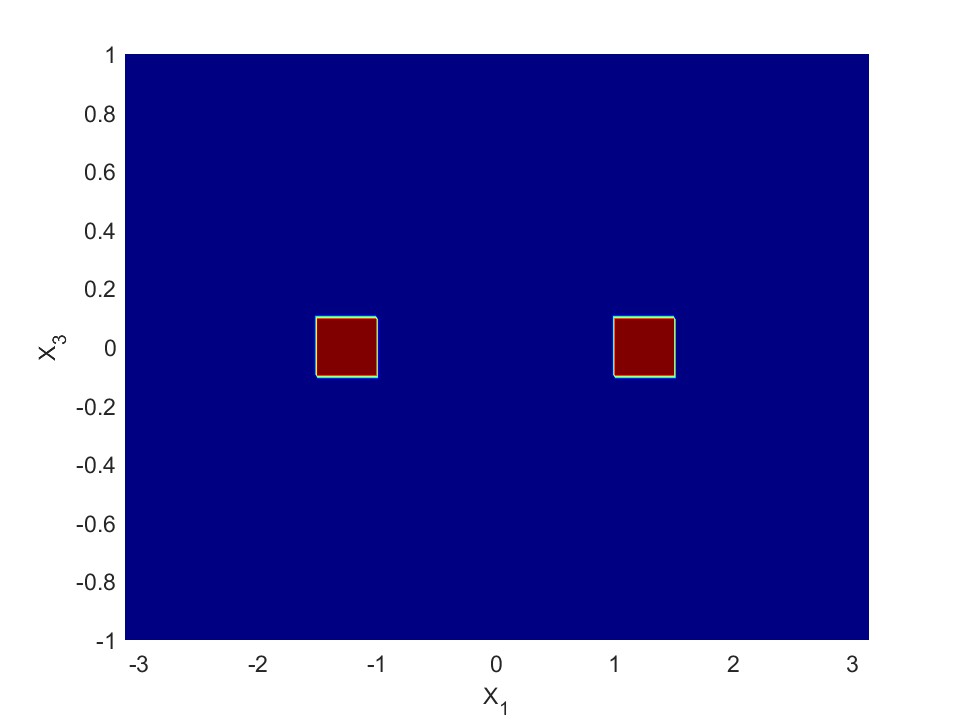}} \hspace{-0.3cm}
\subfloat[]{\includegraphics[width=5.5cm]{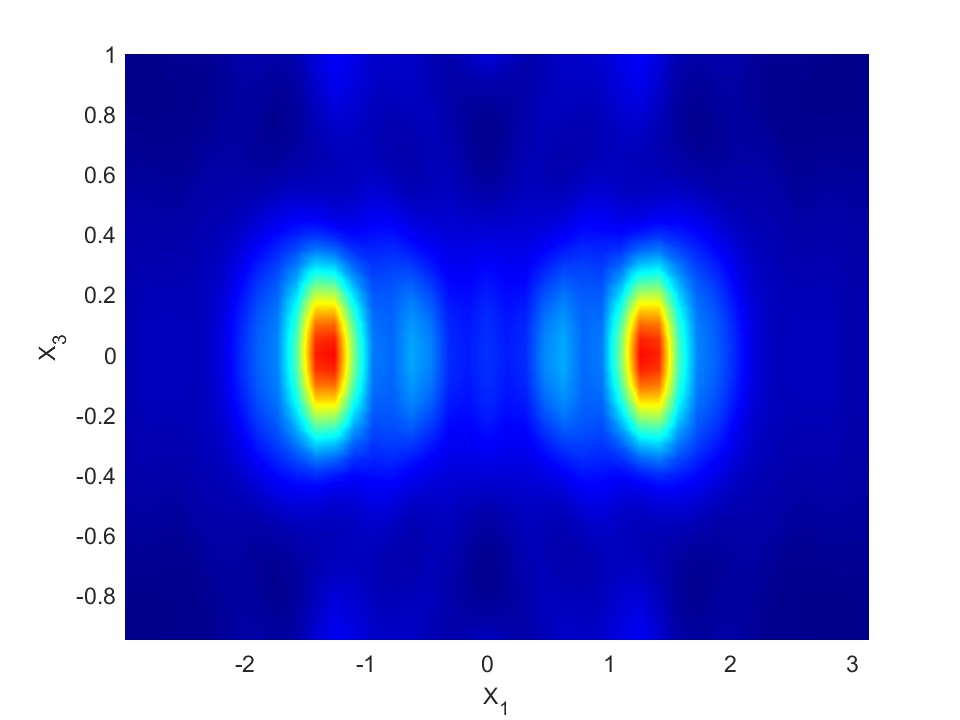}} \hspace{-0.3cm}
\subfloat[]{\includegraphics[width=5.5cm]{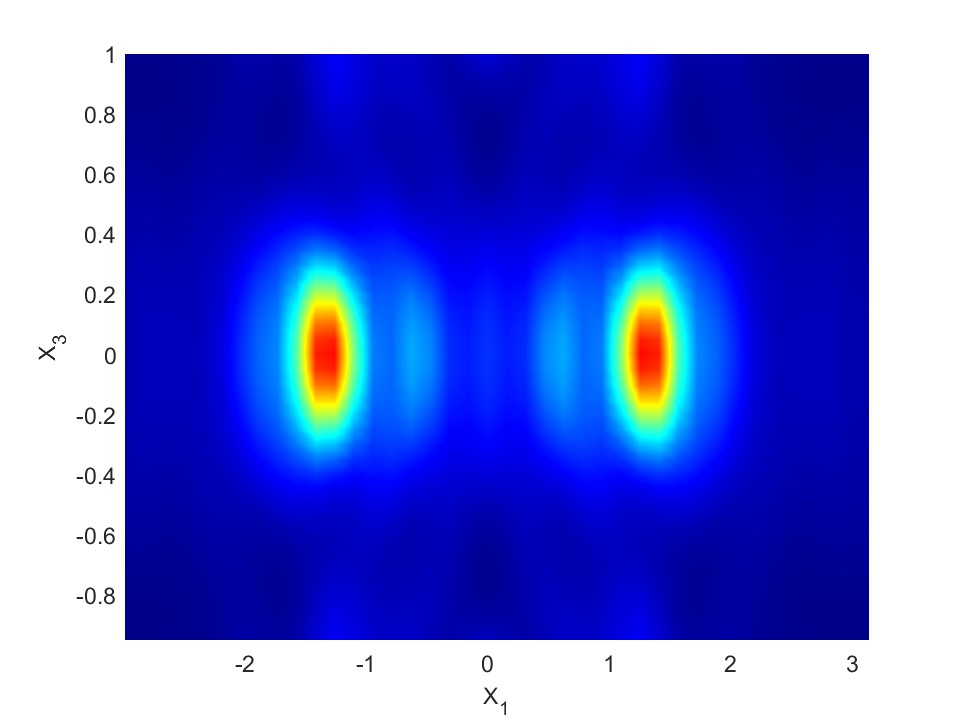}}
\caption{Reconstruction with  and without evanescent modes. First column (a, d, g): True geometry in 3D and 2D views. 
Second column (b, e, h): reconstruction with evanescent modes. Second column (c, f, i): reconstruction without evanescent modes}
\label{evan}
\end{figure}

\subsection{Reconstruction with highly noisy data (Figure~\ref{noise})}
 \label{sub:noise}
We already proved that the imaging  functional is stable against noise in the data, and numerical results in this section further justify that. We tested the performance of the method against different noise levels. In Figure \ref{noise}, we include the reconstructions of the ring at $\delta = 40\%$ and $\delta = 60\%$. Along with its reconstruction at $\delta = 20\%$ in Figure \ref{compare_ring}, we can see that the method gives very similar results. It is also known that the factorization method which was widely studied 
for imaging periodic media is not very stable against noise in the data (see, e.g., \cite{Arens2005}). Thus this test  numerically justifies the fact that the sampling method studied in this paper is more stable than  the factorization method.
\begin{figure}[h!!!]
\centering
\subfloat[]{\includegraphics[width=5.5cm]{ring_true_iso}}\hspace{-0.4cm}
\subfloat[]{\includegraphics[width=5.5cm]{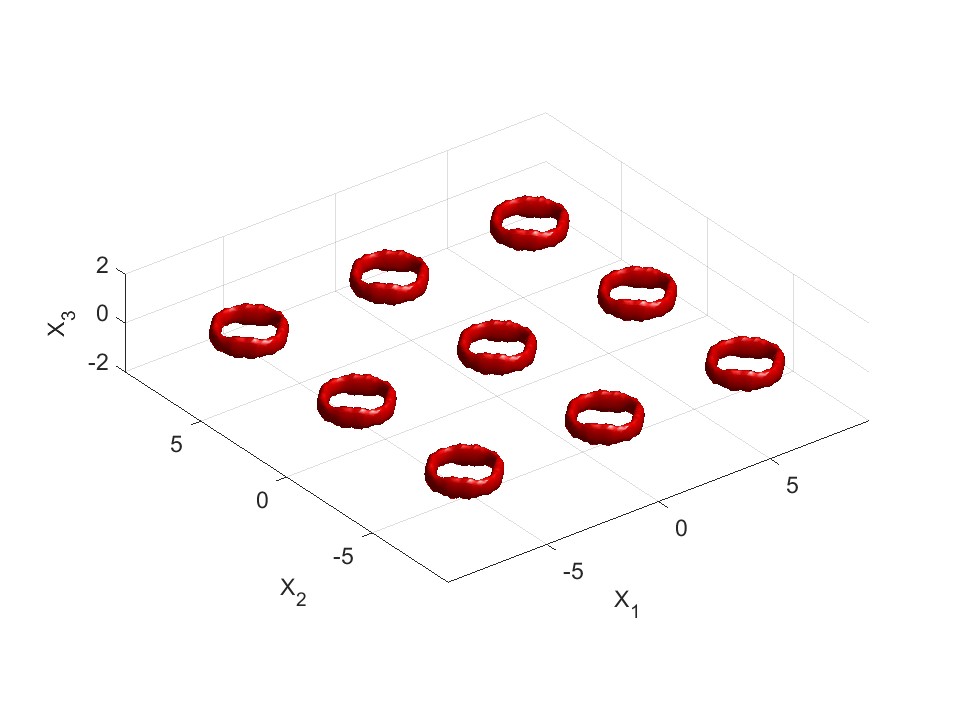}} \hspace{-0.4cm}
\subfloat[]{\includegraphics[width=5.5cm]{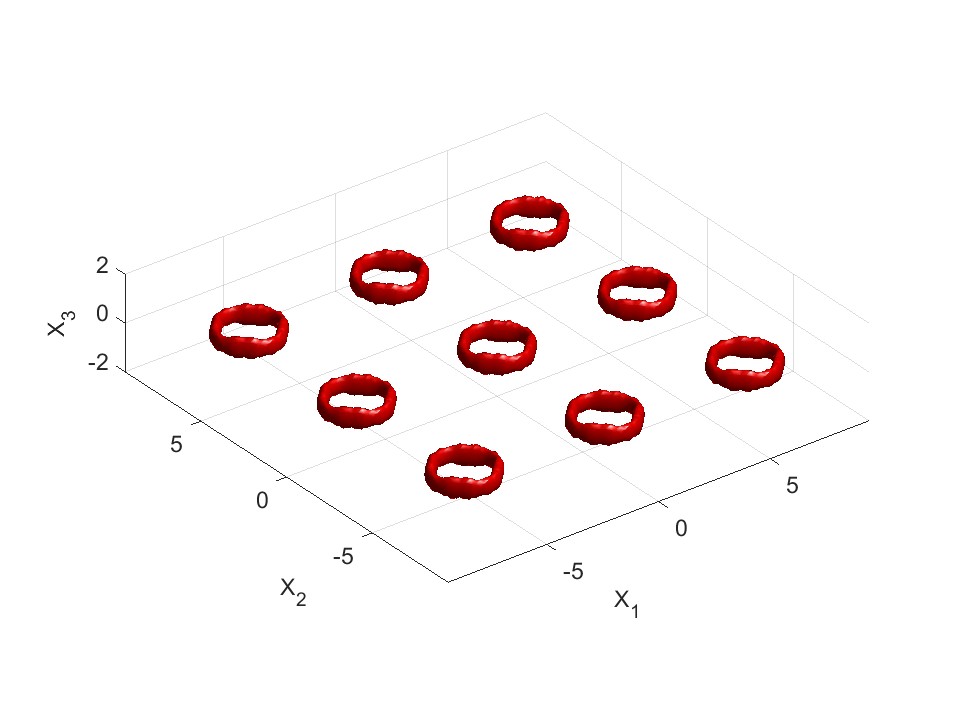}} \\
\subfloat[]{\includegraphics[width=5.5cm]{ring_true_xy}} \hspace{-0.3cm}
\subfloat[]{\includegraphics[width=5.5cm]{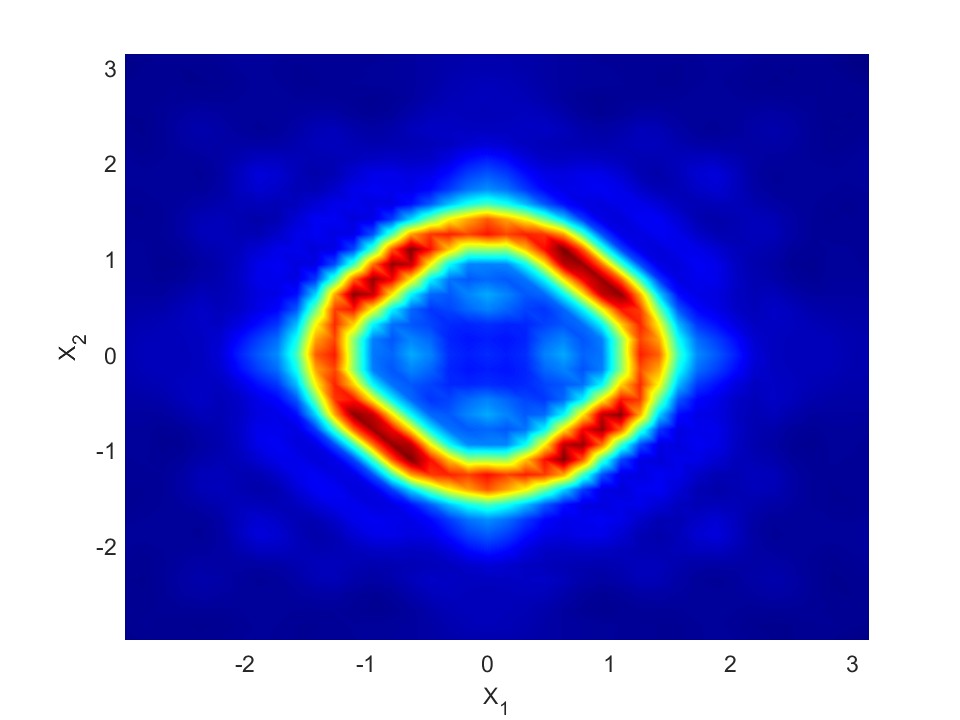}} \hspace{-0.3cm}
\subfloat[]{\includegraphics[width=5.5cm]{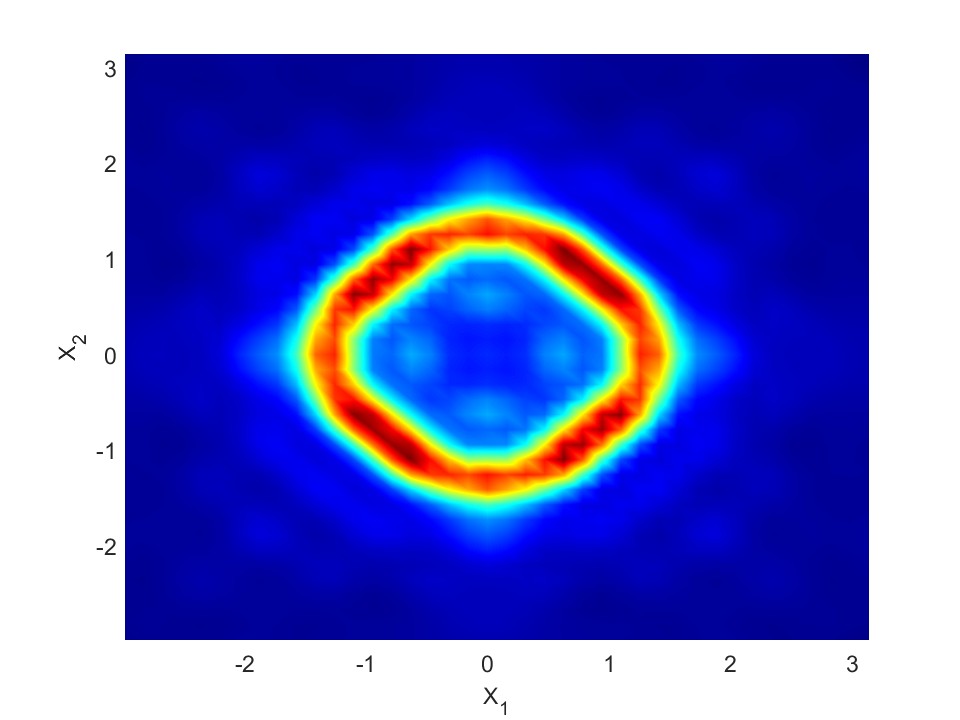}}\\
\subfloat[]{\includegraphics[width=5.5cm]{ring_true_xz}} \hspace{-0.3cm}
\subfloat[]{\includegraphics[width=5.5cm]{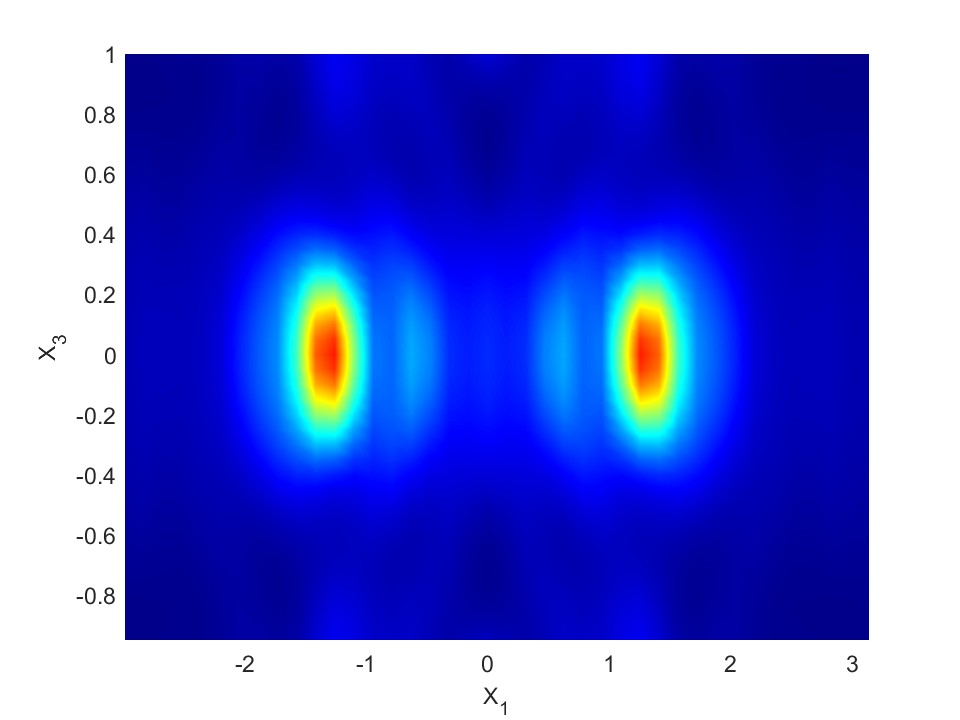}} \hspace{-0.3cm}
\subfloat[]{\includegraphics[width=5.5cm]{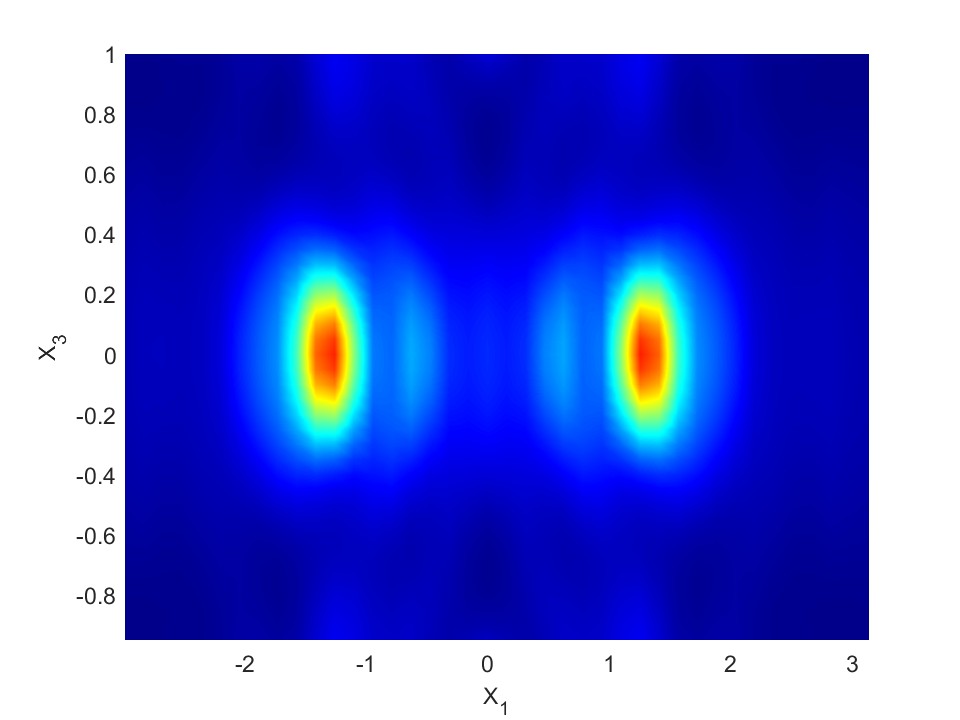}}
\caption{Reconstruction with highly noisy data. First column (a, d, g): True geometry in 3D and 2D views. 
Second column (b, e, h): reconstruction with $40\%$ noise in the data. Second column (c, f, i): reconstruction with $60\%$ noise in the data.}
\label{noise}
\end{figure}

\subsection{Reconstruction with different numbers of incident sources (Figure~\ref{incident})}
\label{source}
Generally, the more incident sources the better the reconstruction. However, when the number of incident sources has reached a certain amount, the reconstruction will not change even if we increase this number. Figure \ref{incident} shows the reconstructions  with $200$ and $800$ incident sources. We can observe that the reconstruction with $800$ incident sources is very similar to that with $450$ incident sources in Figure \ref{compare_ring}, and they are both better than the reconstruction with just $200$ incident sources.

\begin{figure}[h!]
\centering
\subfloat[]{\includegraphics[width=5.5cm]{ring_true_iso}}\hspace{-0.4cm}
\subfloat[]{\includegraphics[width=5.5cm]{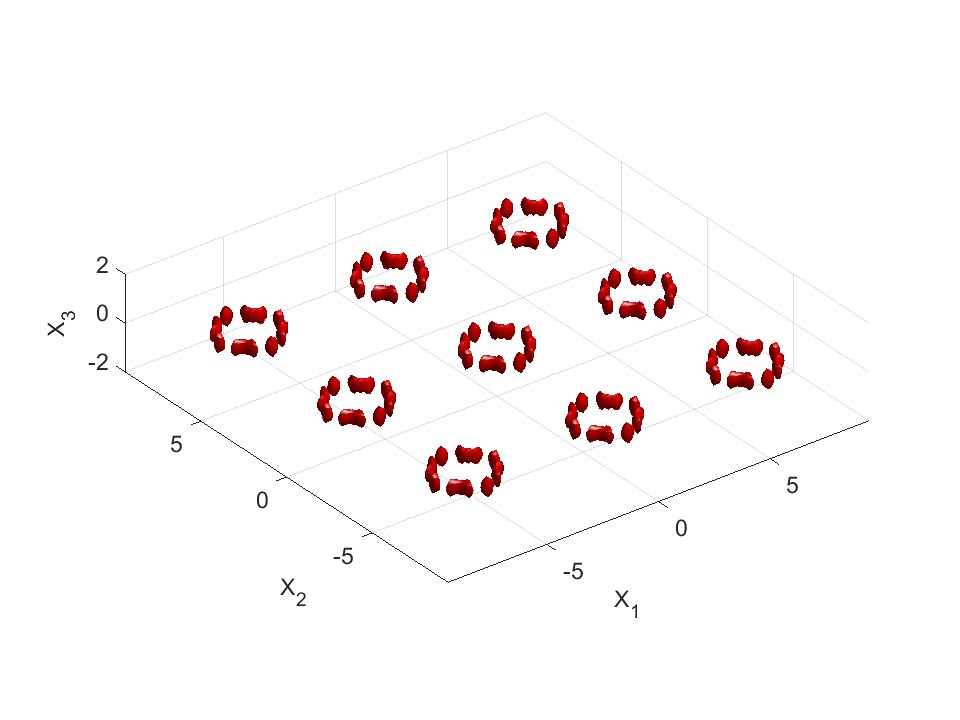}} \hspace{-0.4cm}
\subfloat[]{\includegraphics[width=5.5cm]{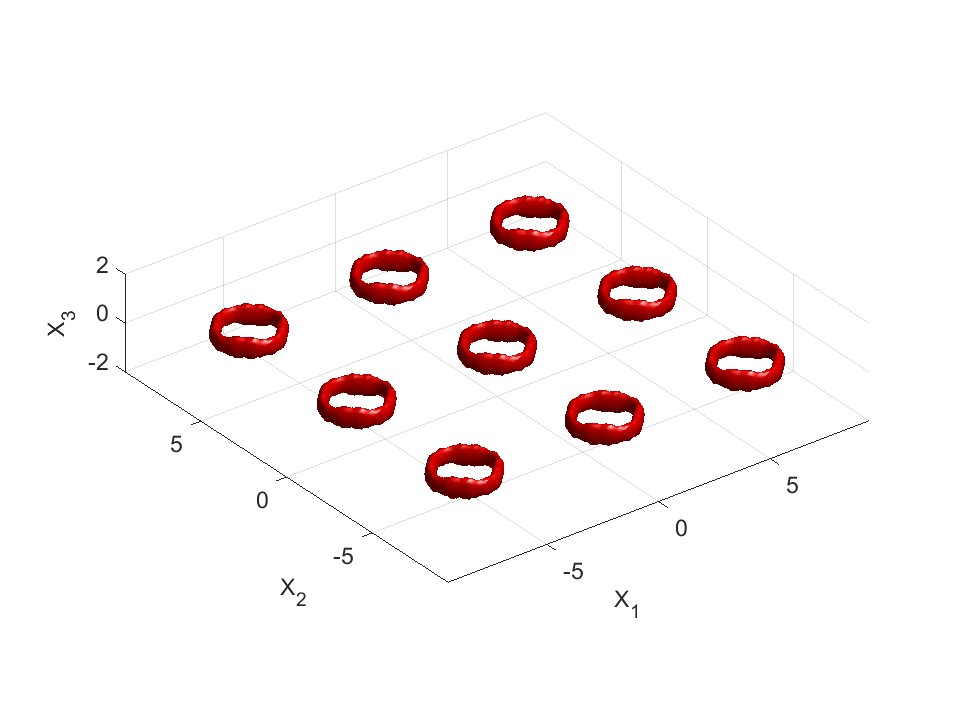}} \\
\subfloat[]{\includegraphics[width=5.5cm]{ring_true_xy}} \hspace{-0.3cm}
\subfloat[]{\includegraphics[width=5.5cm]{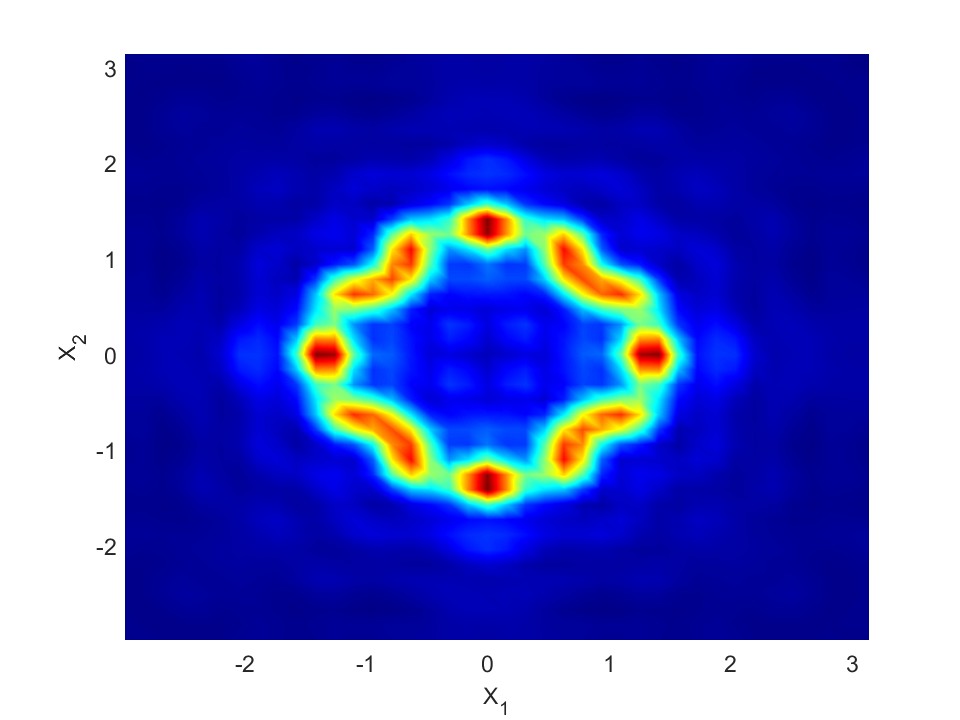}} \hspace{-0.3cm}
\subfloat[]{\includegraphics[width=5.5cm]{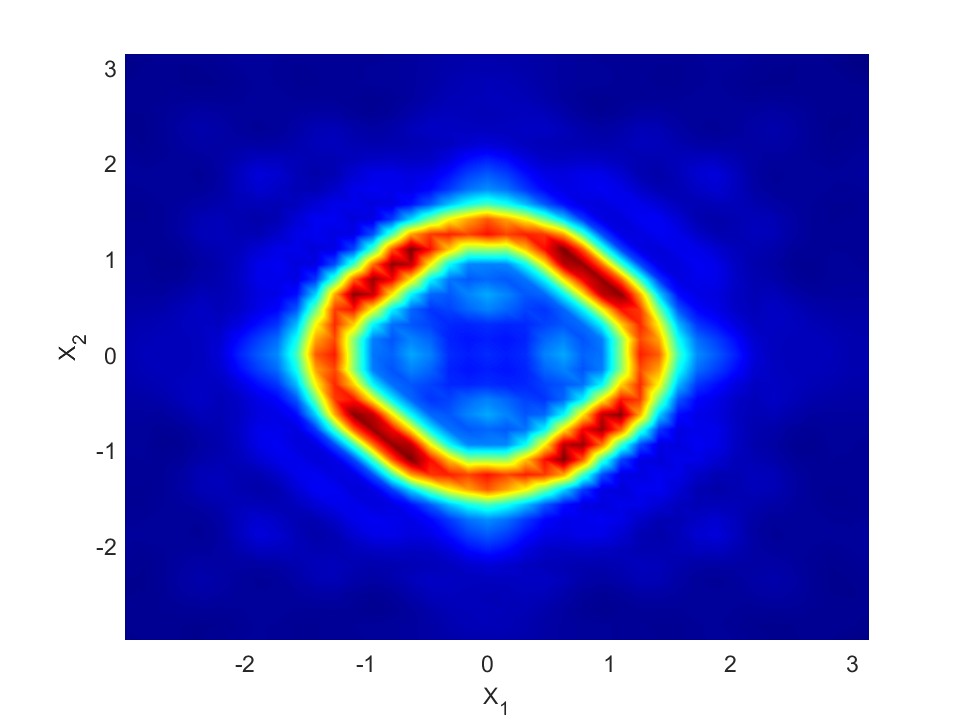}}\\
\subfloat[]{\includegraphics[width=5.5cm]{ring_true_xz}} \hspace{-0.3cm}
\subfloat[]{\includegraphics[width=5.5cm]{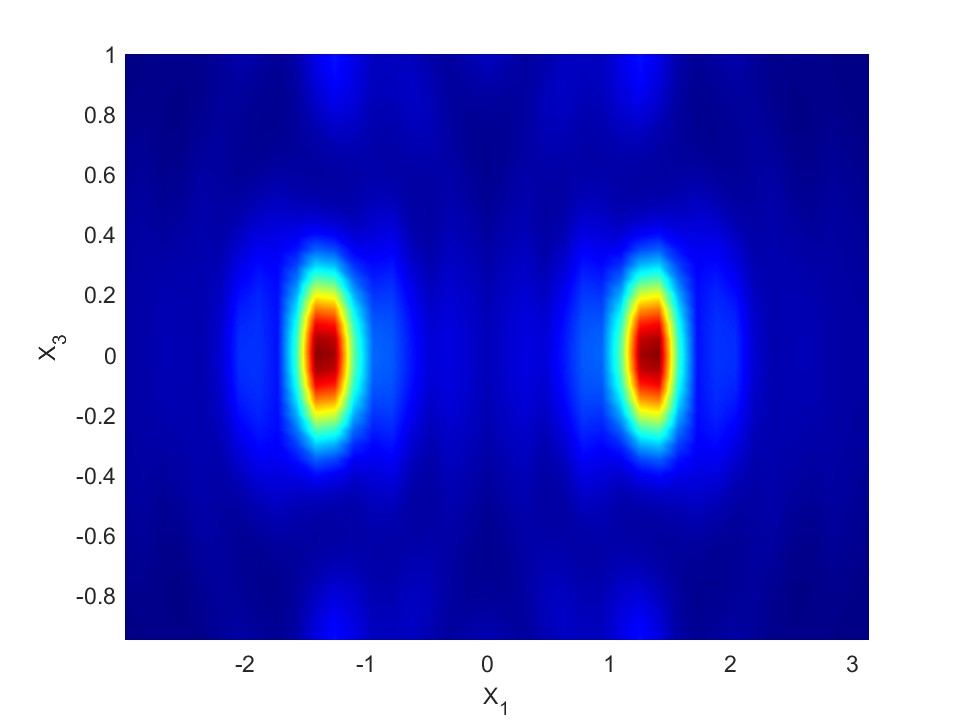}} \hspace{-0.3cm}
\subfloat[]{\includegraphics[width=5.5cm]{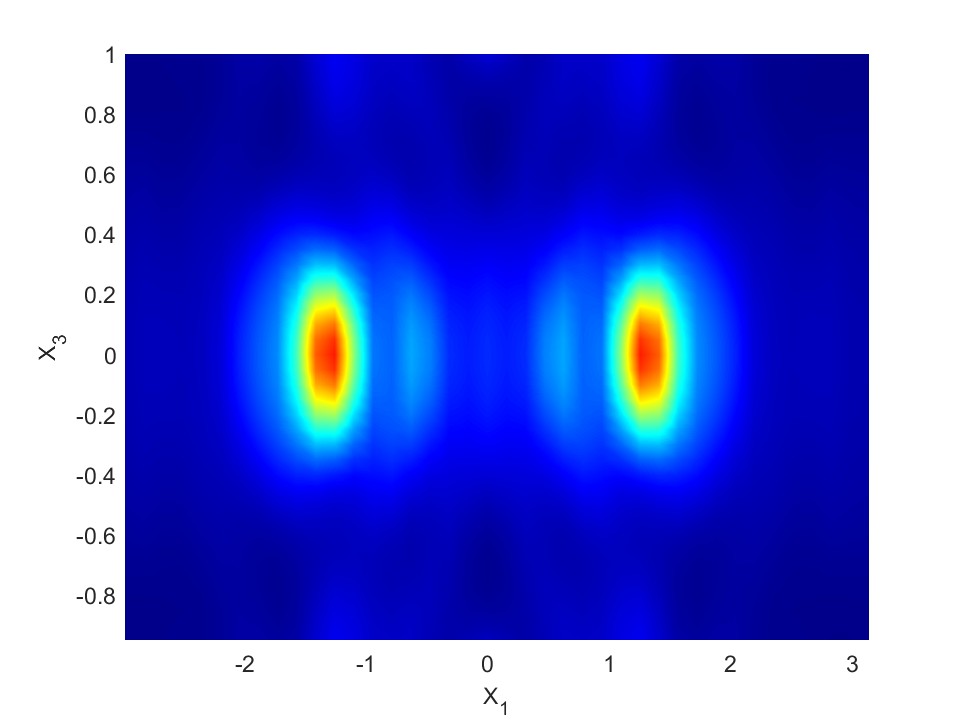}}
\caption{Reconstruction with different numbers of incident sources. First column (a, d, g): True geometry in 3D and 2D views. 
Second column (b, e, h): reconstruction with 200 incident sources. Second column (c, f, i): reconstruction with 800 incident sources.}
\label{incident}
\end{figure}

\subsection{Comparison with the orthogonality sampling method (Figures~\ref{compare_ring}-\ref{compare_cube})}
The orthogonality sampling method (OSM) is a well-known sampling  method that was studied extensively for the case of bounded scattering objects. Here we did multiple comparisons of our proposed method and the OSM. The imaging functional of the OSM is given by
$$
\mathcal{I}_{\text{OSM}}(\z) := \sum_{l=1}^N \left| \int_{\Gamma_{\rho}\cup\Gamma_{-\rho}} \u(\x,l) \cdot \mathbf{q} \overline{G(\z,\x)}\ \d s(\x) \right|^p,
$$
where  the polarization $\mathbf{q}=(1,1,1)^\top$, $\rho = 1.5$, and $p = 3$. 
From Figures~\ref{compare_ring}-\ref{compare_cube}  we can see that the new sampling method method can provide better reconstructions than the OSM. 
The OSM is able to provide reasonable reconstructions in the $x_1$ and $x_2$ directions but not the in the $x_3$ direction. 

\begin{figure}[h!]
\centering
\subfloat[]{\includegraphics[width=5.5cm]{ring_true_iso}}\hspace{-0.4cm}
\subfloat[]{\includegraphics[width=5.5cm]{ring_rec_iso}} \hspace{-0.4cm}
\subfloat[]{\includegraphics[width=5.5cm]{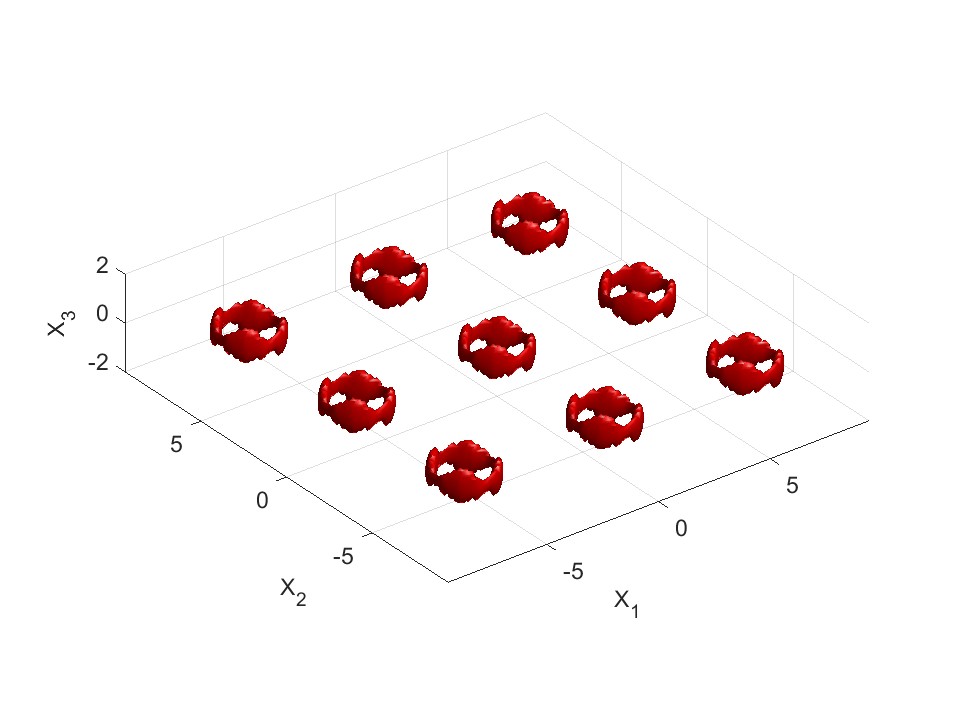}} \\
\subfloat[]{\includegraphics[width=5.5cm]{ring_true_xy}} \hspace{-0.3cm}
\subfloat[]{\includegraphics[width=5.5cm]{ring_rec_xy}} \hspace{-0.3cm}
\subfloat[]{\includegraphics[width=5.5cm]{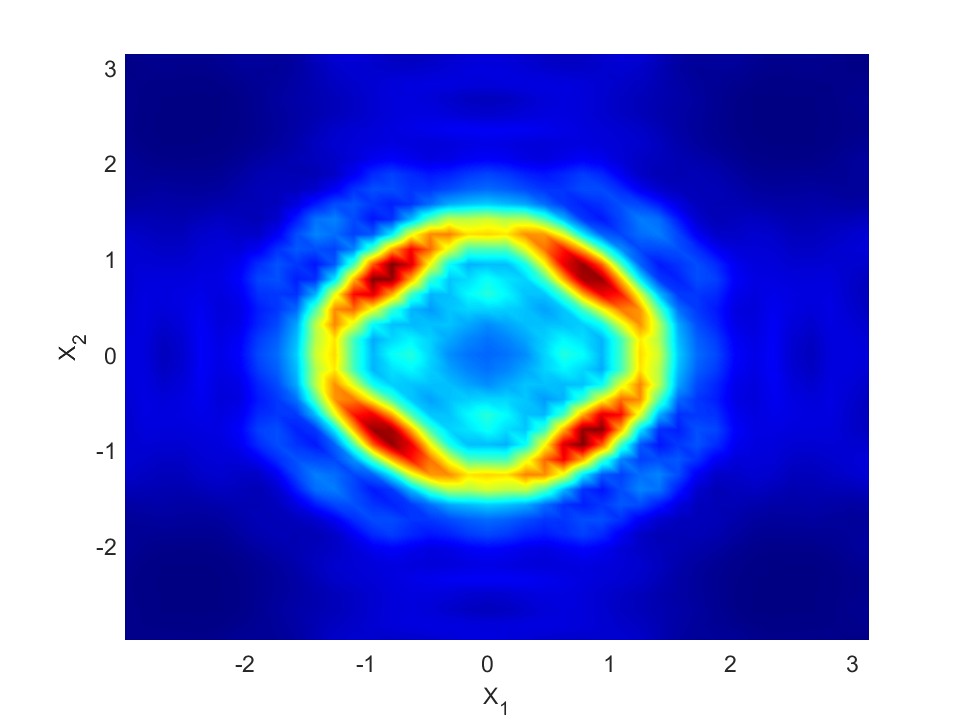}}\\
\subfloat[]{\includegraphics[width=5.5cm]{ring_true_xz}} \hspace{-0.3cm}
\subfloat[]{\includegraphics[width=5.5cm]{ring_rec_xz}} \hspace{-0.3cm}
\subfloat[]{\includegraphics[width=5.5cm]{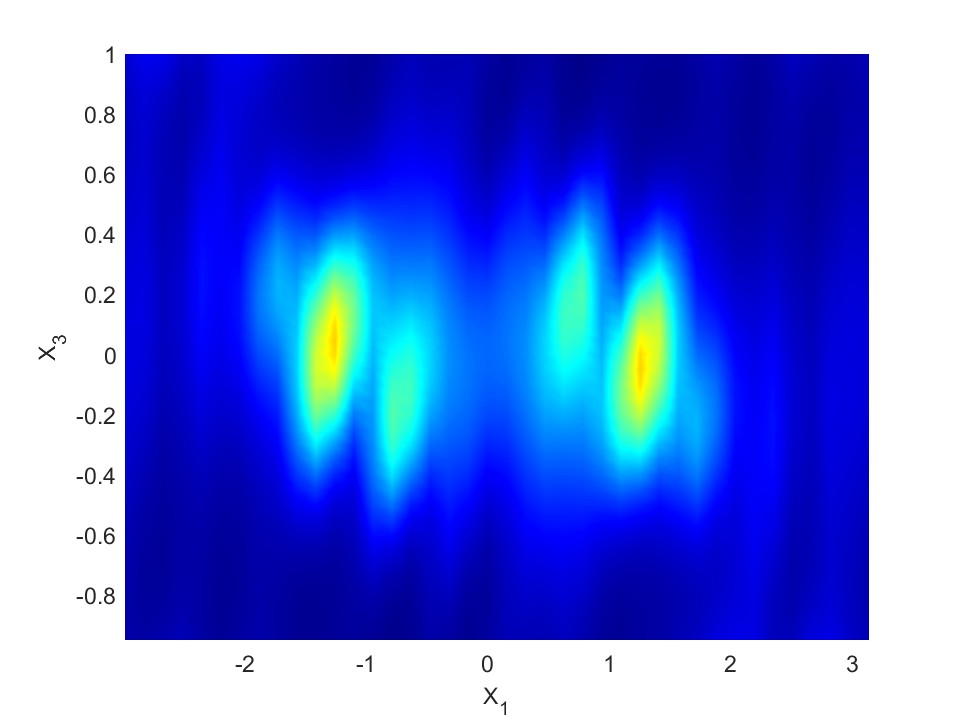}}
\caption{Comparison with the orthogonality sampling method. First column (a, d, g): True geometry in 3D and 2D views. 
Second column (b, e, h): reconstruction with the new sampling method. Second column (c, f, i): reconstruction with the orthogonality sampling method.}
\label{compare_ring}
\end{figure}

\begin{figure}[h!]
\centering
\subfloat[]{\includegraphics[width=5.5cm]{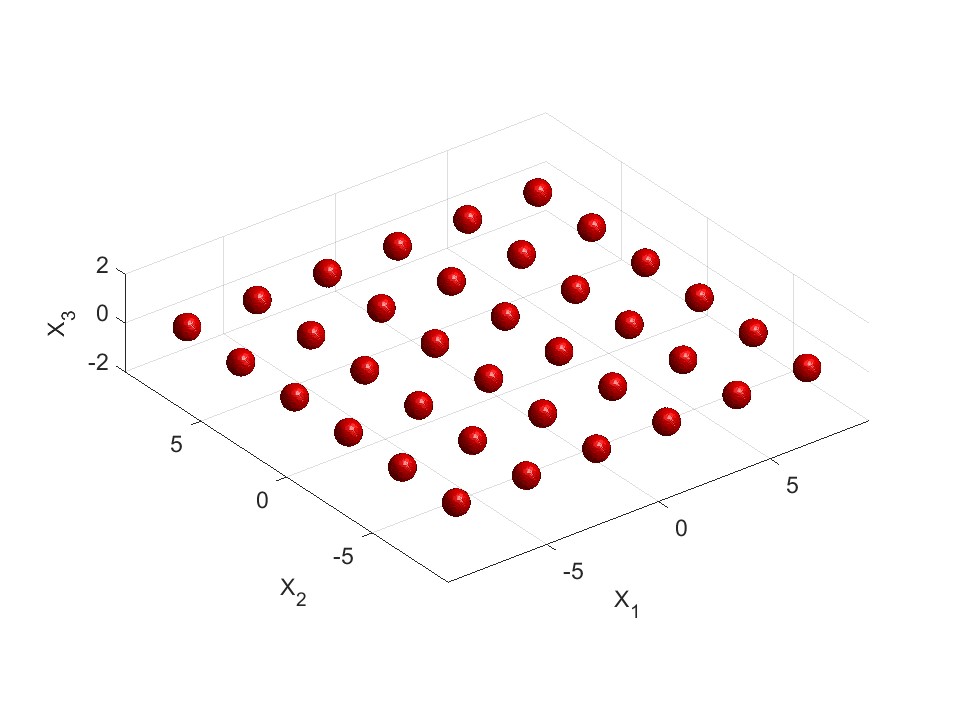}}\hspace{-0.4cm}
\subfloat[]{\includegraphics[width=5.5cm]{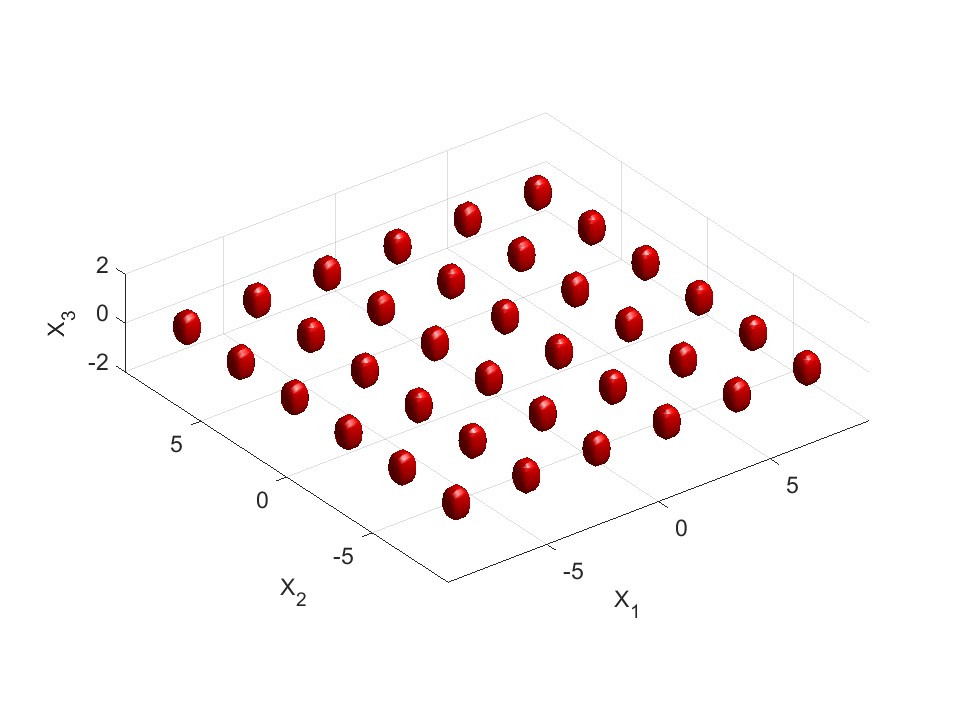}} \hspace{-0.4cm}
\subfloat[]{\includegraphics[width=5.5cm]{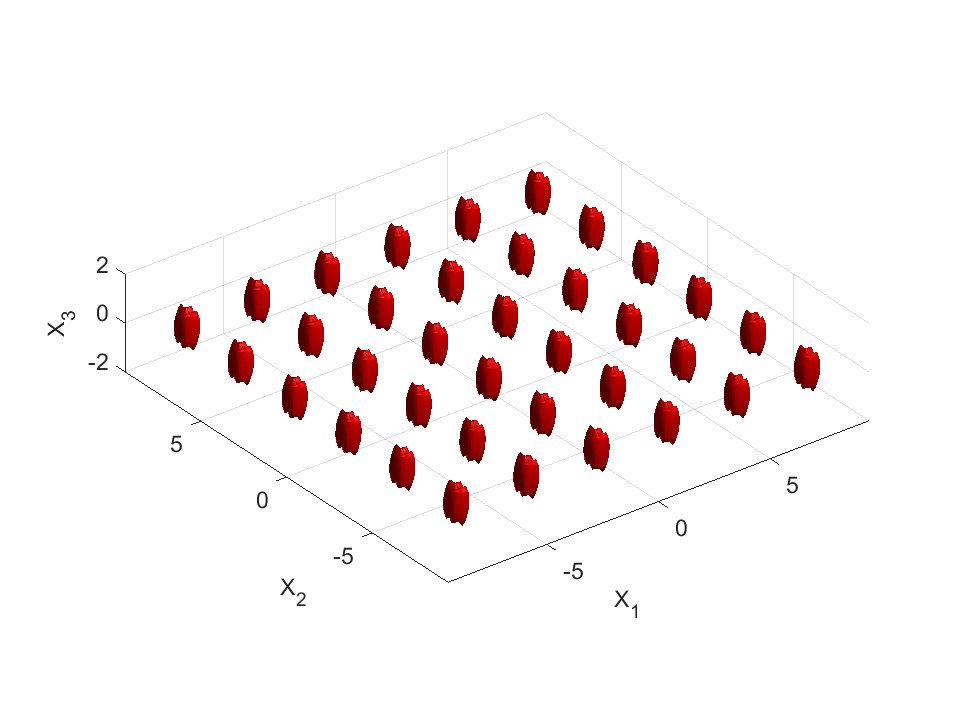}} \\
\subfloat[]{\includegraphics[width=5.5cm]{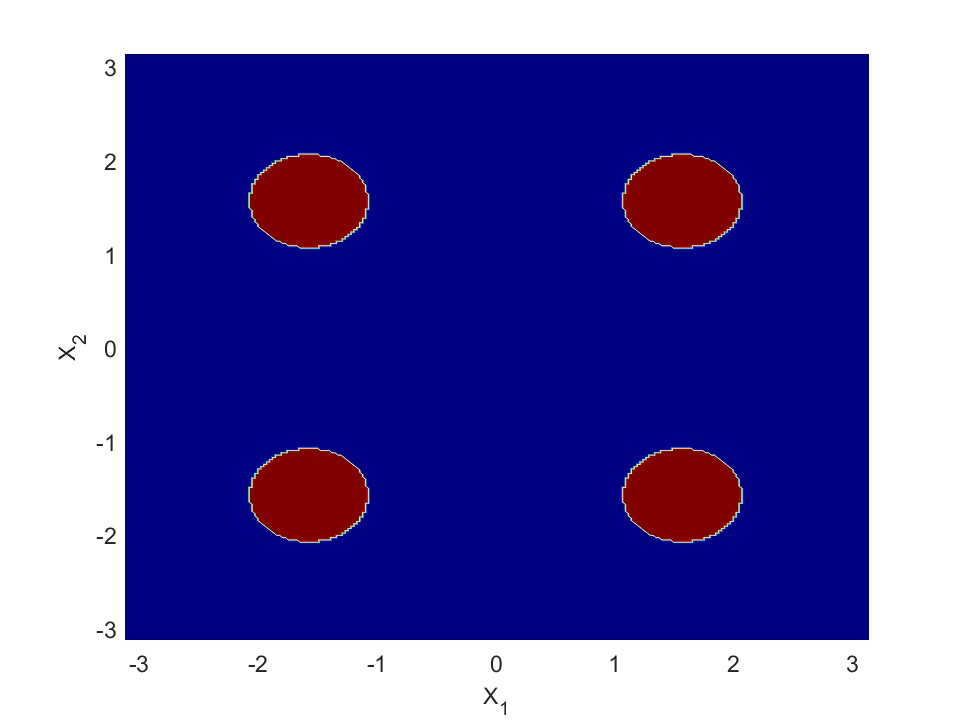}} \hspace{-0.3cm}
\subfloat[]{\includegraphics[width=5.5cm]{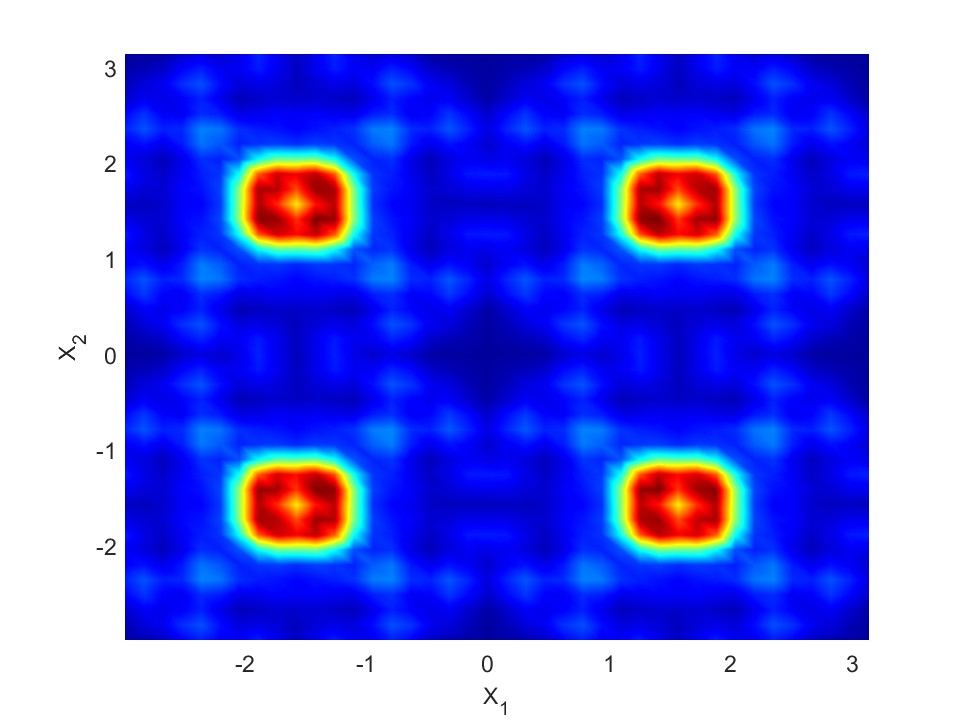}} \hspace{-0.3cm}
\subfloat[]{\includegraphics[width=5.5cm]{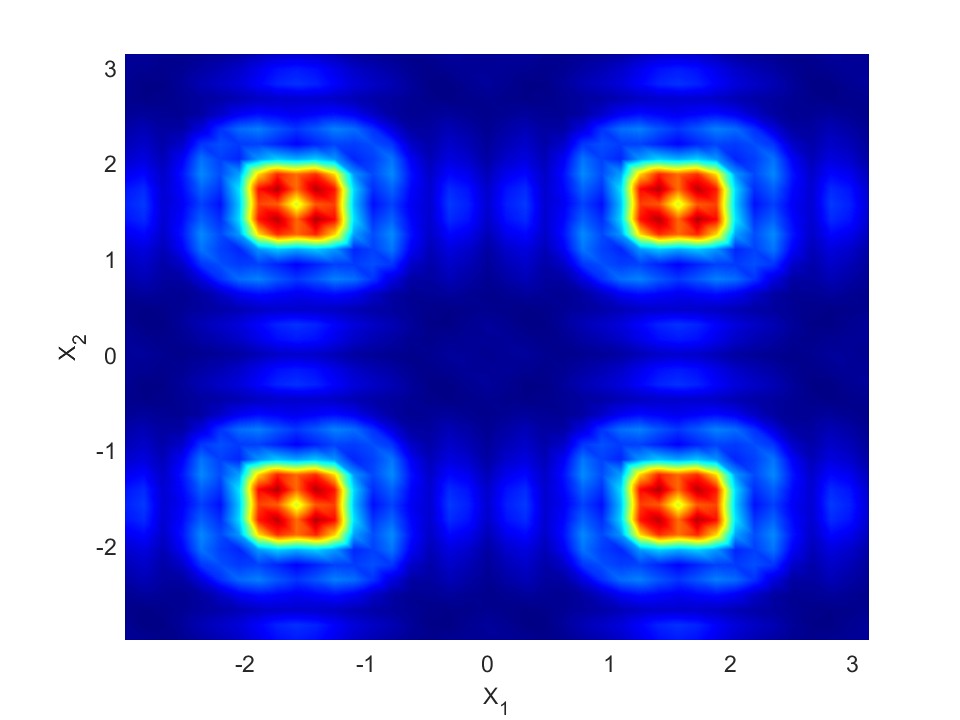}}\\
\subfloat[]{\includegraphics[width=5.5cm]{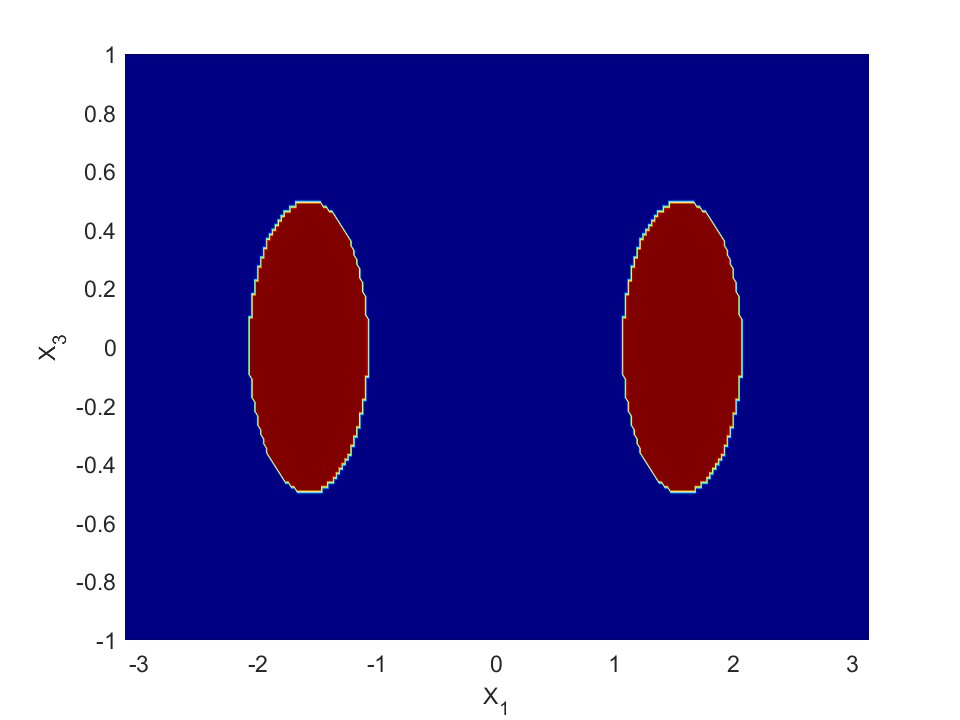}} \hspace{-0.3cm}
\subfloat[]{\includegraphics[width=5.5cm]{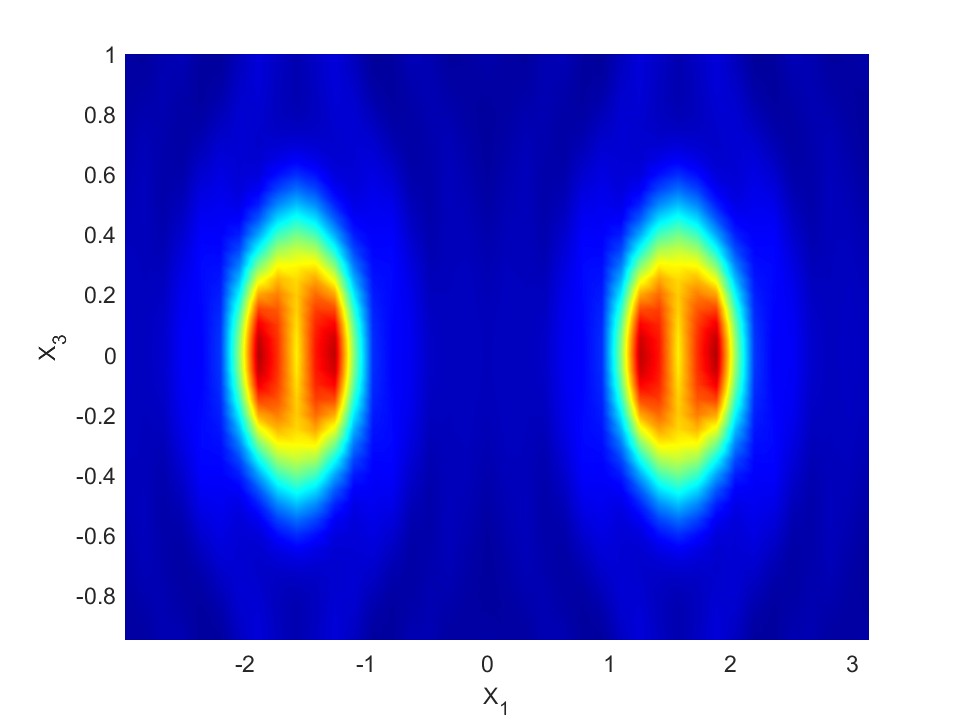}} \hspace{-0.3cm}
\subfloat[]{\includegraphics[width=5.5cm]{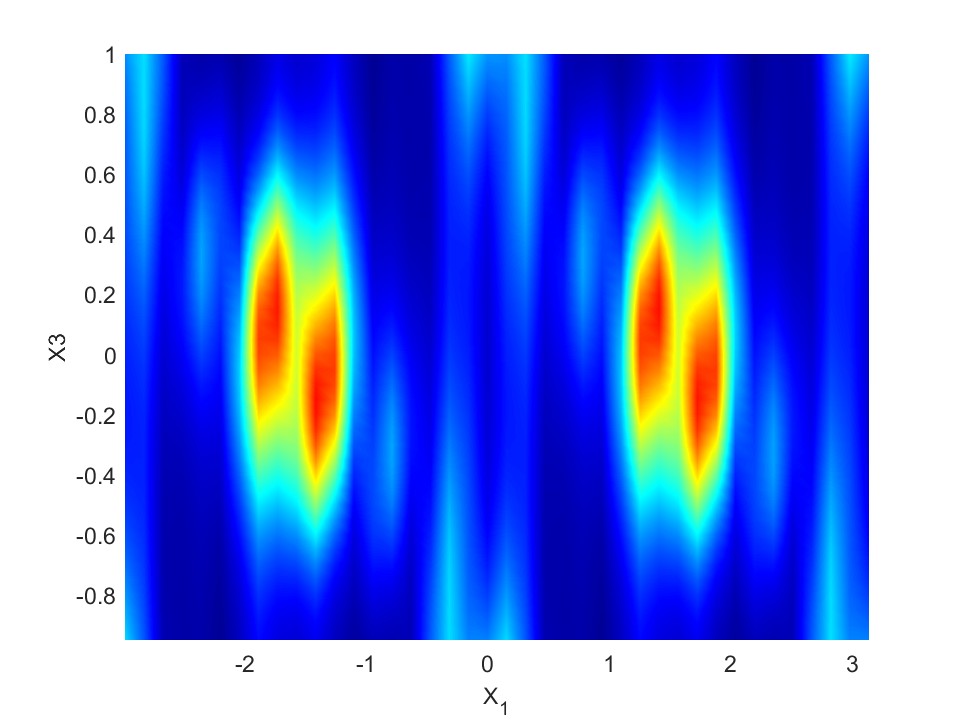}}
\caption{Comparison with the orthogonality sampling method. First column (a, d, g): True geometry in 3D and 2D views. 
Second column (b, e, h): reconstruction with the new sampling method. Second column (c, f, i): reconstruction with the orthogonality sampling method.}
\label{compare_balls}
\end{figure}

\begin{figure}[h!]
\centering
\subfloat[]{\includegraphics[width=5.5cm]{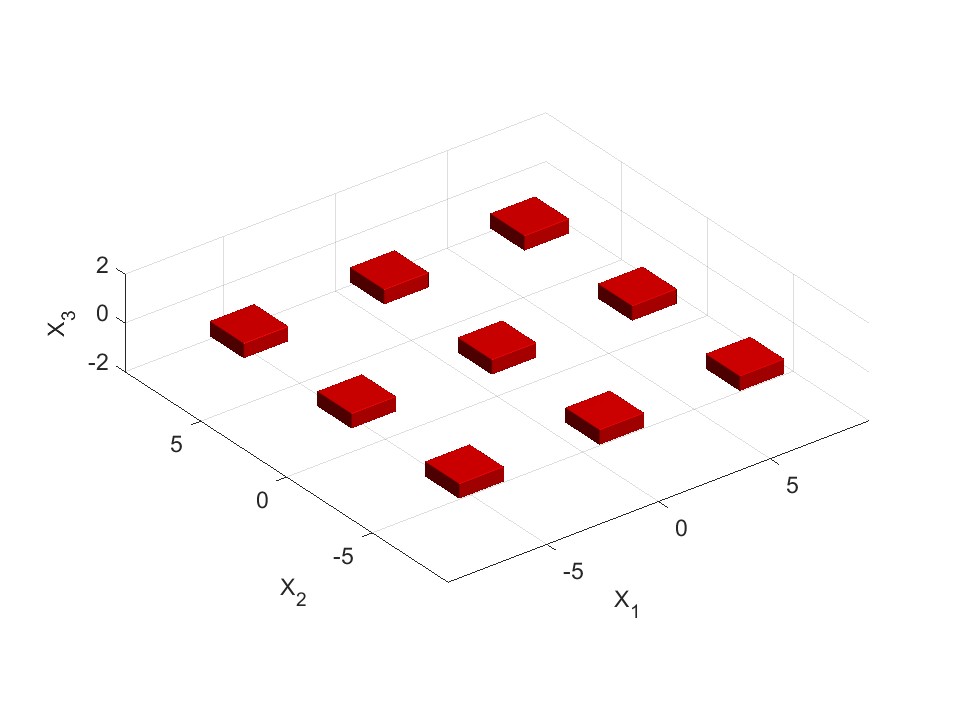}}\hspace{-0.4cm}
\subfloat[]{\includegraphics[width=5.5cm]{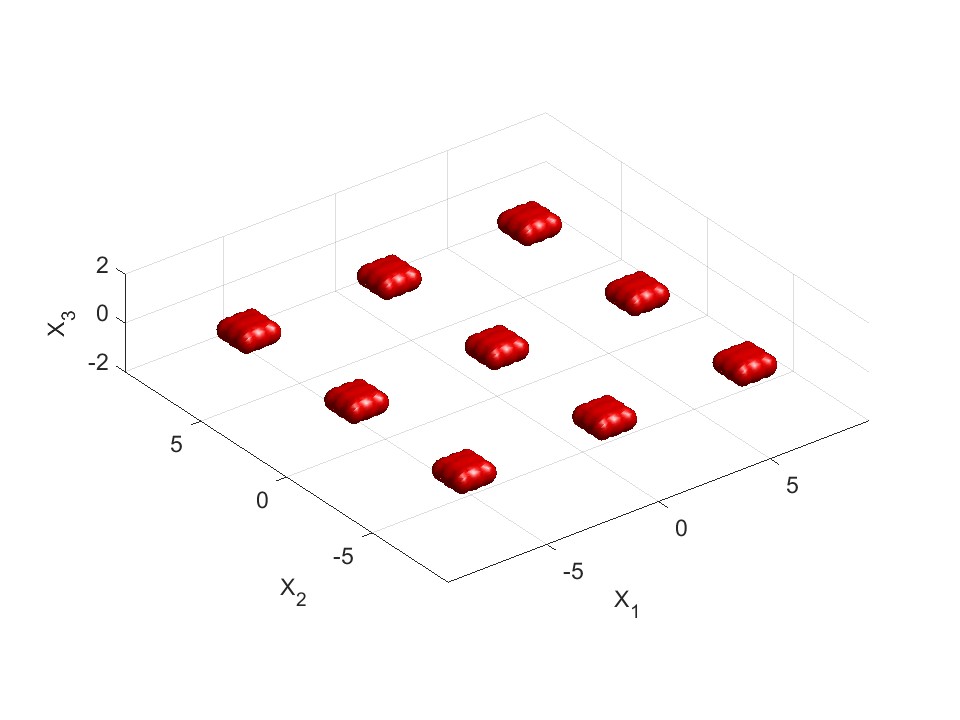}} \hspace{-0.4cm}
\subfloat[]{\includegraphics[width=5.5cm]{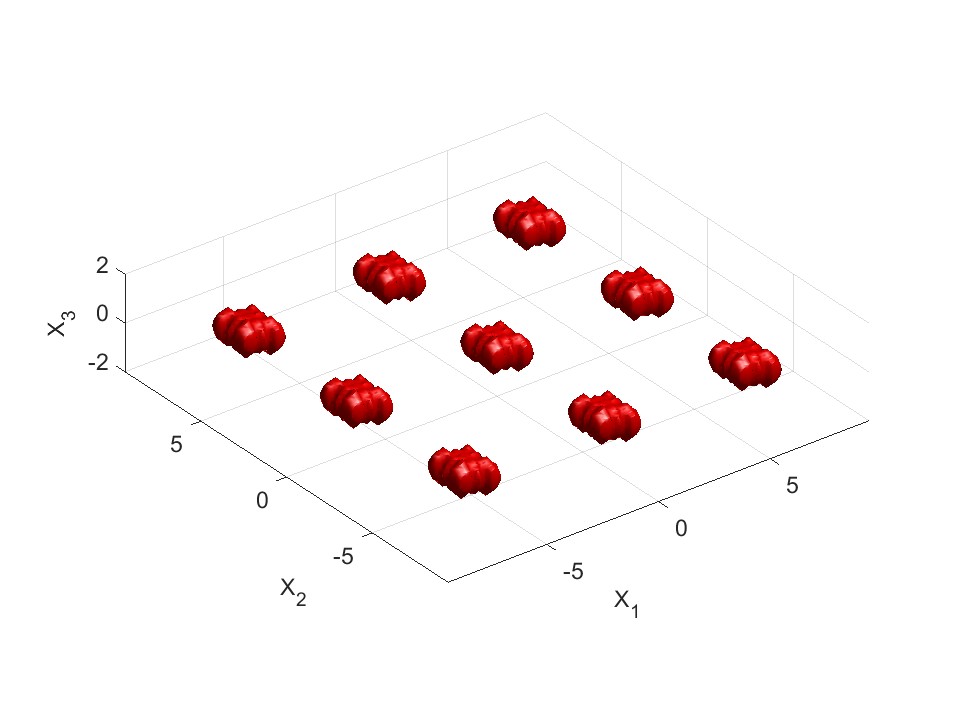}} \\
\subfloat[]{\includegraphics[width=5.5cm]{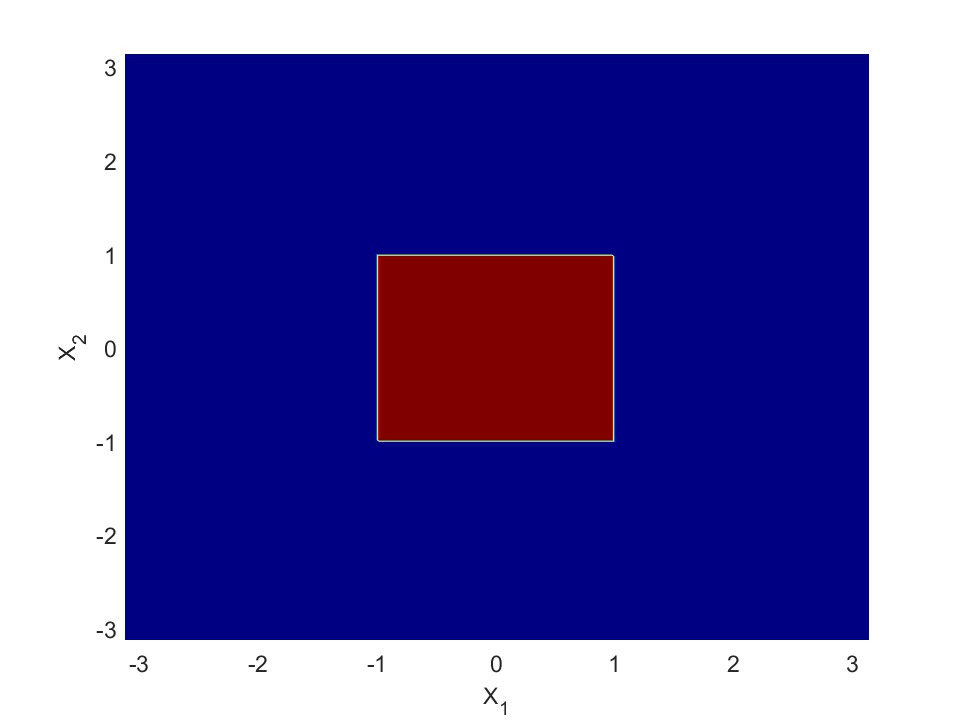}} \hspace{-0.3cm}
\subfloat[]{\includegraphics[width=5.5cm]{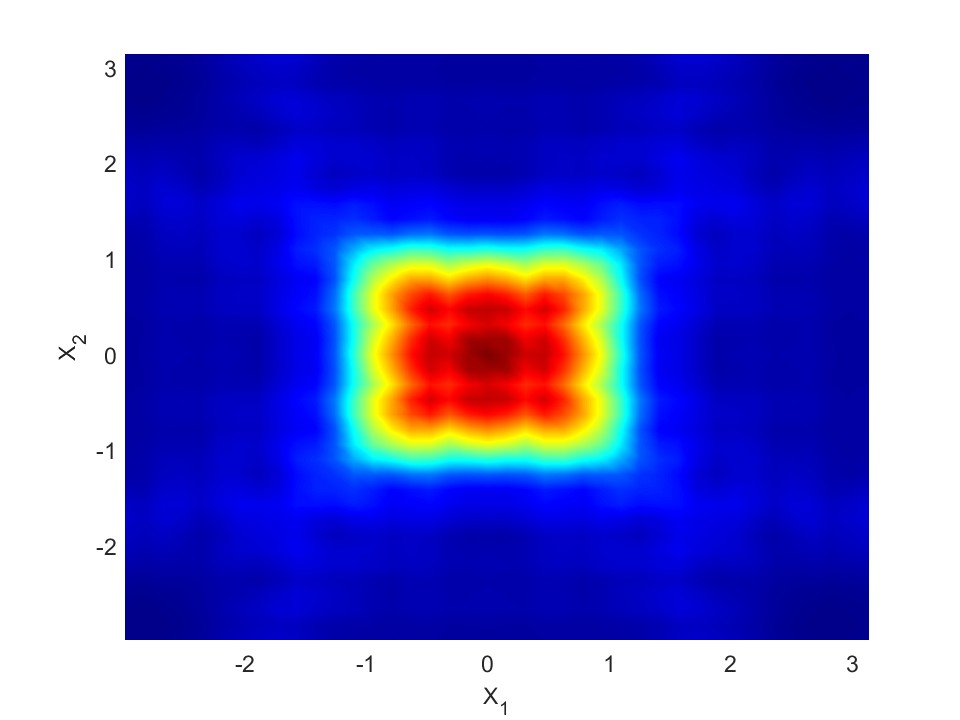}} \hspace{-0.3cm}
\subfloat[]{\includegraphics[width=5.5cm]{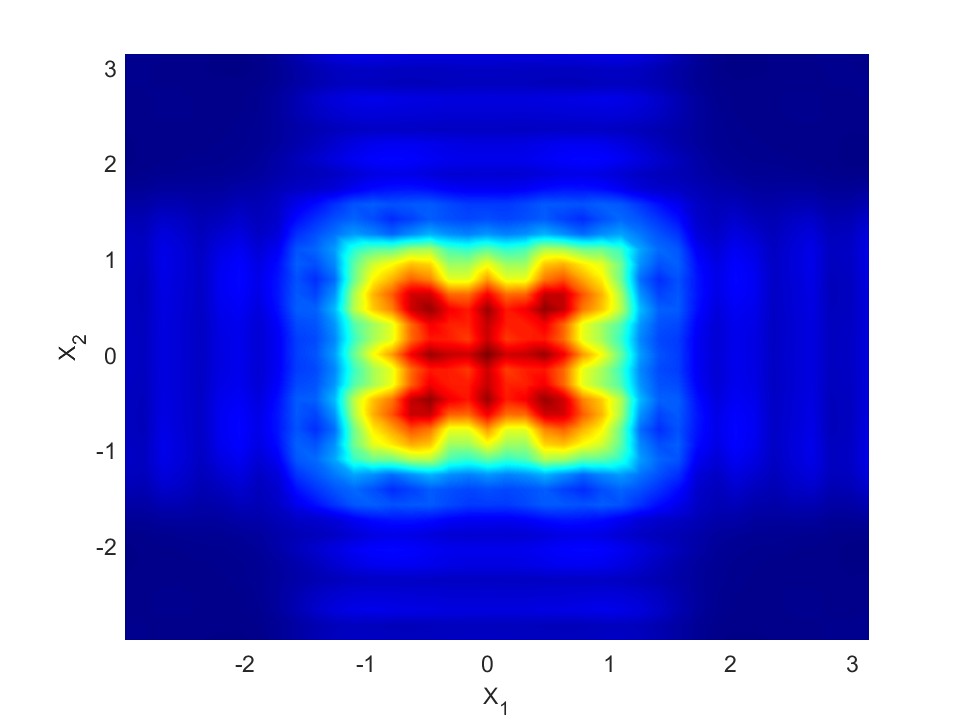}}\\
\subfloat[]{\includegraphics[width=5.5cm]{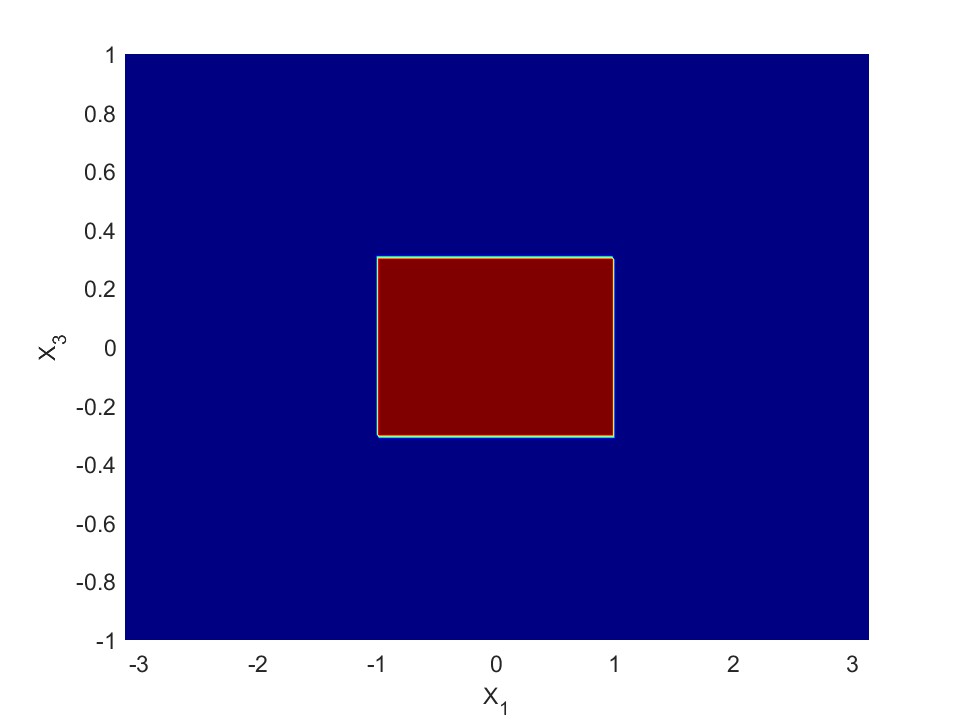}} \hspace{-0.3cm}
\subfloat[]{\includegraphics[width=5.5cm]{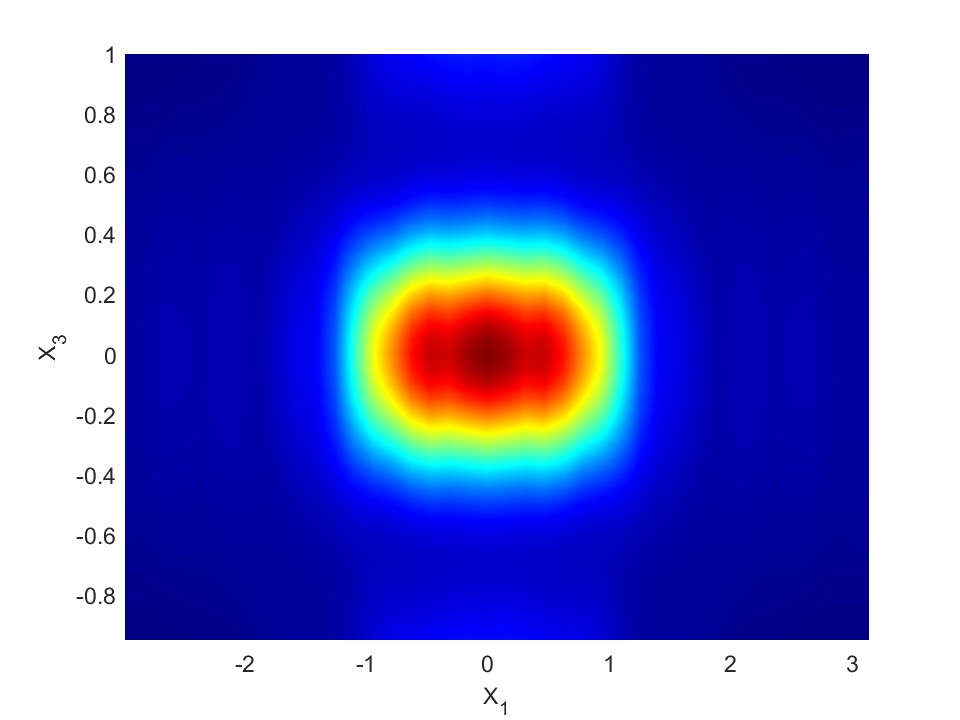}} \hspace{-0.3cm}
\subfloat[]{\includegraphics[width=5.5cm]{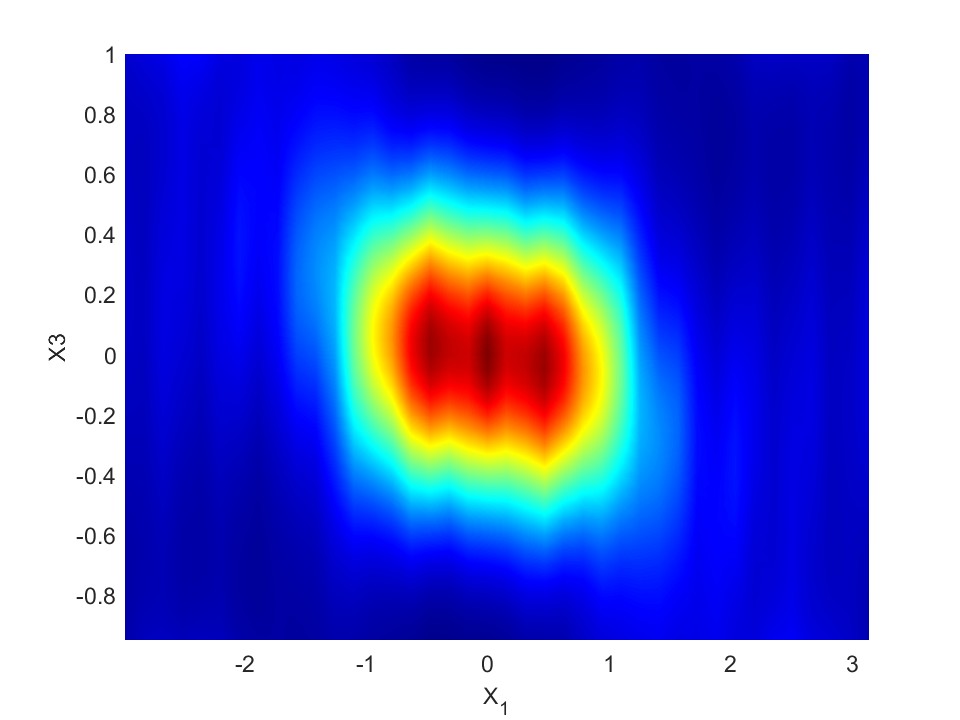}}
\caption{Comparison with the orthogonality sampling method. First column (a, d, g): True geometry in 3D and 2D views. 
Second column (b, e, h): reconstruction with the new sampling method. Second column (c, f, i): reconstruction with the orthogonality sampling method.}
\label{compare_cube}
\end{figure}


\textbf{Acknowledgment.}  The work of the D.-L. Nguyen and T. Truong  was partially supported by NSF Grant  DMS-2208293.

\bibliographystyle{plain}
\bibliography{ip-biblio2}

\end{document}